\documentclass[12pt]{amsart}
\usepackage[utf8]{inputenc}

\usepackage{amsmath,amsfonts,amsthm,amssymb}
\usepackage{graphicx}
\usepackage[matrix,arrow,curve]{xy}
\usepackage[pdftex,unicode]{hyperref}
\usepackage{cmap}
\usepackage{microtype}

\newtheorem{thm}{Theorem}[section]

\newtheorem{prop}[thm]{Proposition}
\newtheorem{defn}{Definition}[section]
\newtheorem{cor}[thm]{Corollary}

\title{The space of framed chord diagrams as a Hopf module}
\author{Maksim Karev}
\thanks{Research is supported in part by the RFBR grant 13-01-00383a.}
\address{St.-Petersburg division of V.A. Steklov institute of Mathematics. 191023, Fontanka 27, Saint-Petersburg, Russia.}
\email{max.karev@gmail.com}

\begin{document}

\begin{abstract}
This note is dedicated to the study of a Hopf module structures on the space of framed chord diagrams and framed graphs. We also introduce a framed version of the chromatic polynomial and propose two methods to construct framed weight systems.
\end{abstract}

\maketitle
\section*{Introduction}

The study of finite-type invariants of immersions of the circle into the plane was initiated by V. Arnold. In his paper (\cite{Arnold}), he gave a description of the three basic invariants, corresponding three different components of the discriminant subspace in the space of the immersions: the closure of the immersions with a direct self-tangency ($J^+$), the closure of the immersions with an inverse self-tangency ($J^-$), and the closure of the immersions with a triple point ($St$) (for the details see~\cite{Arnold}). He also noticed, that circle immersions into a plane are closely related to Legendrian knots in a solid torus $S^1\times \mathbb R^2$. Indeed, a projection of a Legendrian knot along the $S^1$ is a plane curve.

We call the union of $J^+$ and $J^-$ components of the discriminant subspace the $J$-discriminant. A finite-type invariant associated with the $J$-discriminant is referred to as $J$-invariant.

 In the paper~\cite{goryunov} V.~Goryunov introduces the notion of a marked chord diagrams, describing the combinatorics of a non-generic Legendrian knot in the solid torus, and presents a universal finite-type invariant for such immersions, which gives rise to the universal finite-type $J^+$-invariant of circle immersions. In his 1997 article, J.~W.~Hill shows that any relation on symbols of finite-type invariants of Legendrian knots without so-called dangerous self-tangencies in a solid torus  is a consequence of certain relations, which are called $\mu$-marked 4T-relations. For the details see~\cite{Hill}.
 
The notion of framed chord diagram (see the definition below) appeared in  an attempt to understand the combinatorics $J$-invariants  in the paper of S.~Lando (see~\cite{Lando}). By an analogy with the knots finite-type invariant theory, a framed chord diagram represents the combinatorial type of a non-generic immersion, admitting simple self-tangencies (both direct and inverse) as singularities. Similarly to the knots case, the symbol of a finite-type invariant of immersions gives rise to a functional on the space of framed chord diagrams. Of course, not every functional comes from a finite-type invariant, and it turns out that those that actually come from a finite-type invariant must satisfy the framed 4T-relations (see the definitions below). Hill's results imply, that the framed 4T-relations are not only necessary, but also sufficient for a functional to come from a $J$-invariant.

The filtered space of knot finite-type invariants is dual to the bialgebra of knots with singular knot filtration (see, for example,~\cite{CDbook}). The induced graded Hopf algebra structure on the space of ordinary chord diagrams implies that the latter is just a polynomial algebra. $J$-invariants admit a natural product, which gives rise to a well-defined coproduct on the vector space of framed chord diagrams. However, it is not clear yet whether the space of framed chord diagrams admits a nice product, which makes the study of the space more difficult.

In this note we partly fix the lack of algebra structure on the space of framed chord diagrams by providing it with a Hopf module structure over the Hopf algebra of (non-framed) chord diagrams.

Weight system (or, simply speaking, linear map from the space of chord diagrams satisfying 4T-relations) is a very important notion in the knots finite-type invariants theory, as, due to the Vassiliev-Kontsevich theorem, it allows to construct knots invariants. We prove, that any weight system admits a lifting to a framed weight system via the discoloration operator (see the definition in the section 3). The comodule structure on the space of framed chord diagrams also gives a way to produce framed weight systems.

The space of framed chord diagrams admits a well-defined map to the space of framed graphs (introduced in~\cite{Lando}), induced by the operation of taking the intersection graph of a diagram. It turns out, that a framed version of the chromatic polynomial can be defined, and also used to construct a framed weight system.

This note is the natural development of the ideas of~\cite{Lando}. It consists of the introduction and three sections. The Hopf module structure on the space of framed chord diagrams is introduced in Section 1. Section 2 is devoted to similar properties of the space of framed graphs. The description of a framed chromatic polynomial, along with two methods of construction of framed weight systems are presented in the Section~3.

The author would like to express his sincere gratitude to S. Duzhin for his support, numerous discussions, and the permission to use pictures from~\cite{CDbook}. I also thank S. Lando for the interest and valuable comments that helped to make the text easier to understand. The work was inspired by the article of V. Kleptsyn and E. Smirnov ``Plane curves and bialgebra of Lagrangian subspaces.'' (\cite{KS}).

In this note, $\mathbb K$ always denotes a chosen field of characteristics 0.

\section{Framed chord diagrams}

Let us first recall the definition of the algebra of chord diagrams $\mathcal A$.

\begin{defn}
Let $n\ge 0$. An \emph{order $n$ chord diagram} is a collection of $n$ chords~---  pairs of
pairwise distinct points on an oriented circle, considered up to orientation preserving diffeomorphisms of the circle.
\end{defn}

In the figures below, the points in a pair forming a chord are drawn connected by an arc.

\begin{defn}
The space $\mathcal A$ is the space spanned by the set of the chord diagrams modulo the \emph{4T-relations}, shown in Figure~\ref{4T} (a).
\end{defn}

The pictures in Figure~\ref{4T} should be understood as follows. It is assumed that the dashed part of the circle may contain other chords' endpoints, but their combinatorial structure is the same for all the four diagrams entering the relation. Only the chords whose endpoints are inside solid intervals are shown explicitly. One of the explicitly shown chords is again the same for all the diagrams involved. We refer to the chord that is not shared by all the diagrams in the relation as the \emph{jumping chord}.

The space $\mathcal A$ is endowed with a natural product $m_\mathcal A\colon \mathcal A \otimes \mathcal A\to \mathcal A$, a natural coproduct $\Delta_\mathcal A\colon \mathcal A\otimes \mathcal A$, the unit $e_\mathcal A\colon \mathbb K \to \mathcal A$, the counit $\epsilon_\mathcal A\colon \mathcal A\to \mathbb K$, and the antipode $\mathcal S_\mathcal A\colon \mathcal A\to \mathcal A$, making $\mathcal A$ into a commutative cocommutative Hopf algebra. The multiplication is given by glueing two diagrams, and the coproduct is given by summing up the tensor products of pairs of chord diagrams formed by a decomposition of the set of chords into two complimentary subsets. For the precise definitions of the maps, we refer the reader to~\cite{CDbook}, section 4.4.

We define the space $\mathcal M$ of framed chord diagrams following Lando~\cite{Lando}.

\begin{defn}
Let $n\ge 0$. An \emph{order $n$ framed chord diagram} is a chord diagram along with a fixed \emph{framing}. A framing is a map from the set of chords to $\mathbb Z/2\mathbb Z$. The preimages of $0$ under this map are called \emph{oriented chords}, the preimages of $1$ are called \emph{disorienting chords}.
\end{defn}

An example of a framed chord diagram is shown in Figure~\ref{fcd}.

\begin{figure}[ht]
\begin{center}
\includegraphics[width=1.5cm]{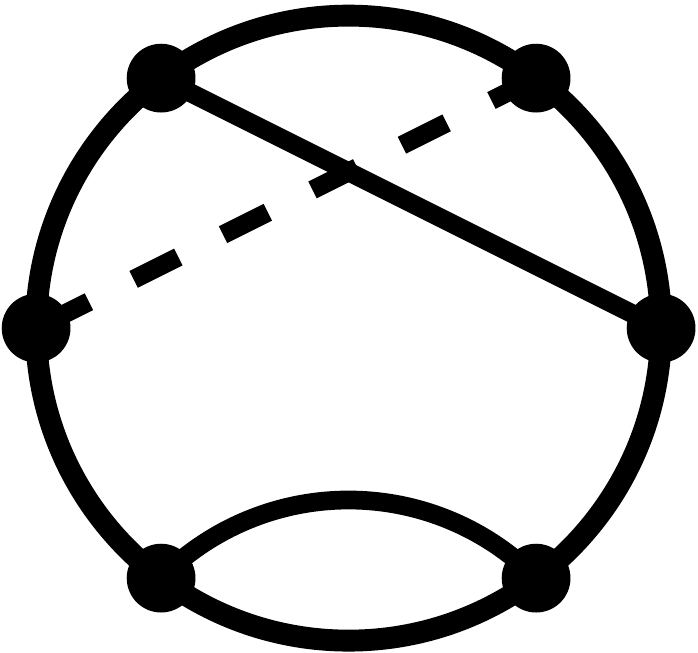}
\end{center}
\caption{An order 3 framed chord diagram. The orientation of the circle (here and everywhere below) obeys the blackboard convention. Oriented chords are drawn as solid arcs, disorienting chord is drawn as dashed arc.}\label{fcd}
\end{figure}

\begin{defn}
A framed chord diagram is said to be \emph{white} if all its chords are disorienting. A framed chord diagram is said to be \emph{black} if all its chords are oriented.
\end{defn}

We consider framed chord diagrams as generating elements of a $\mathbb K$-vector space $M$ over $\mathbb K$.

\begin{defn}
The space $\mathcal M$ of framed chord diagrams is the quotient space of~$M$ modulo the relations shown in figure~\ref{4T} (a) and (b).
\begin{figure}
\begin{center}
\includegraphics[width=1.5cm]{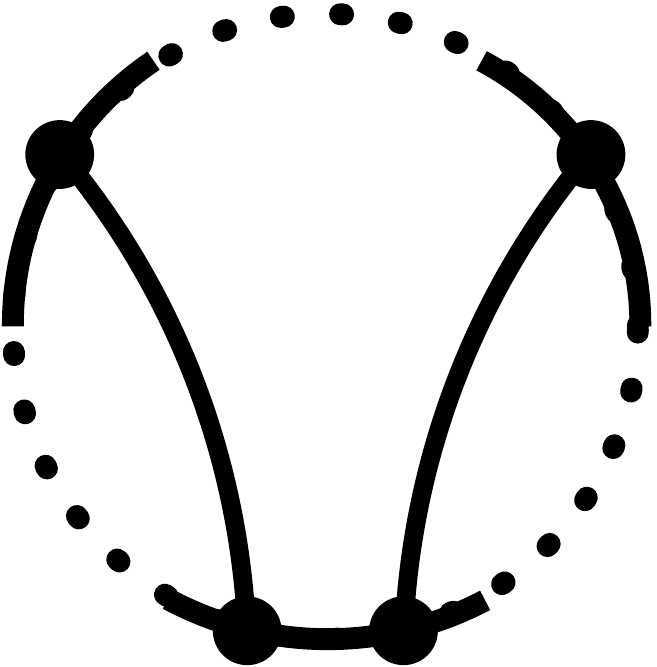} \raisebox{0.7cm}{--} \includegraphics[width=1.5cm]{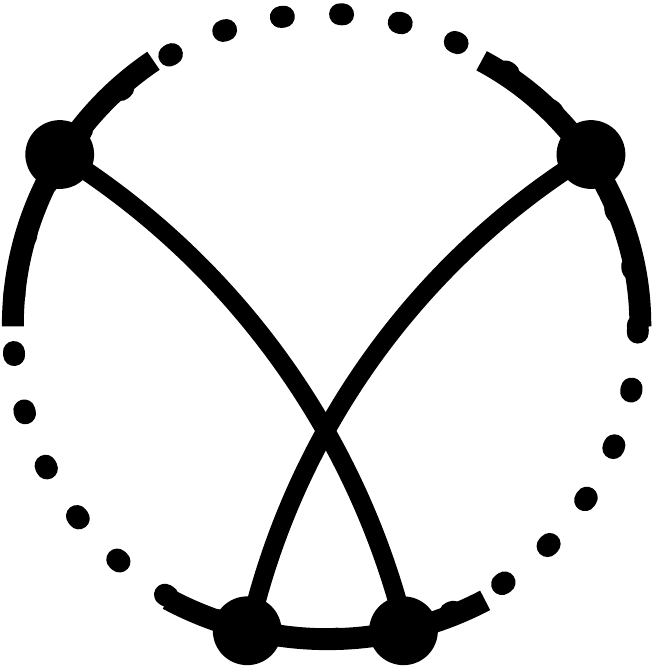} \raisebox{0.7cm}{=} \includegraphics[width=1.5cm]{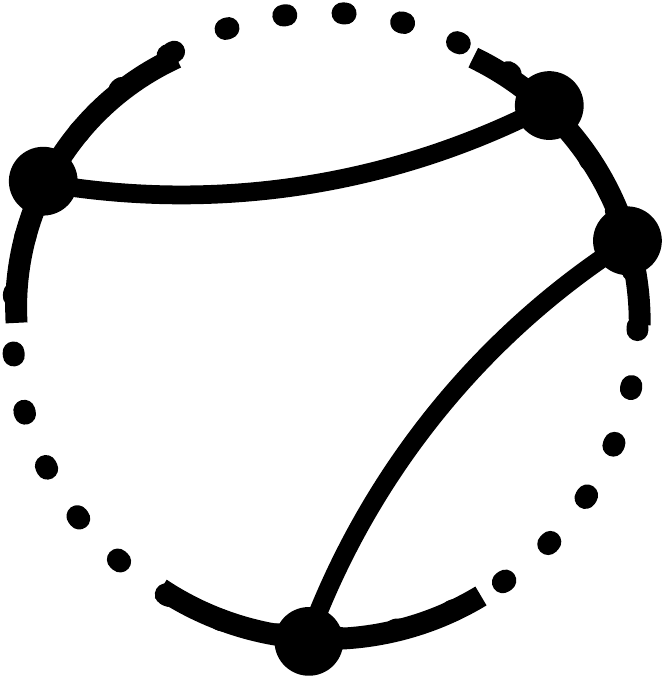} \raisebox{0.7cm}{--} \includegraphics[width=1.5cm]{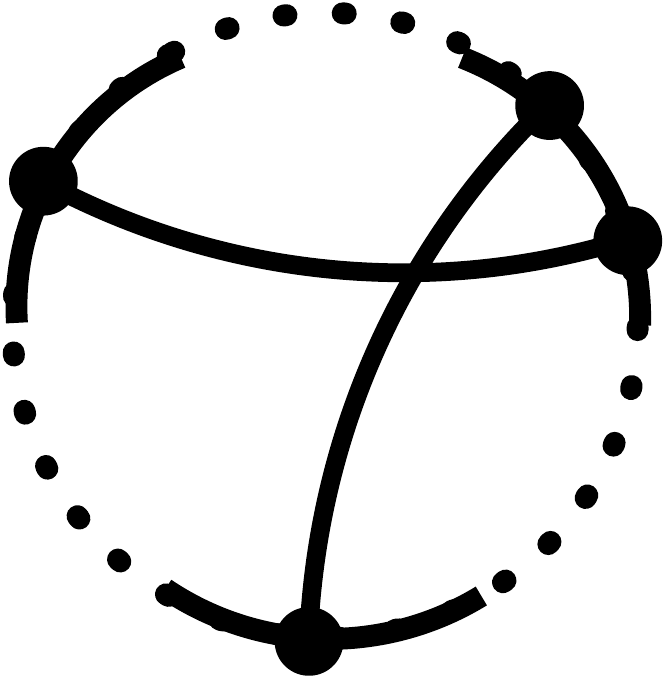}
\vspace{0.3cm}
(a)

\includegraphics[width=1.5cm]{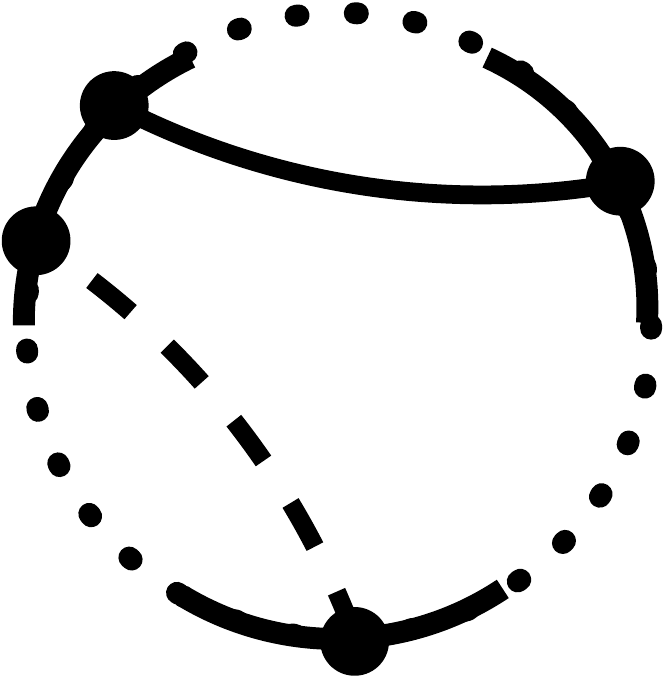} \raisebox{0.7cm}{--}
\includegraphics[width=1.5cm]{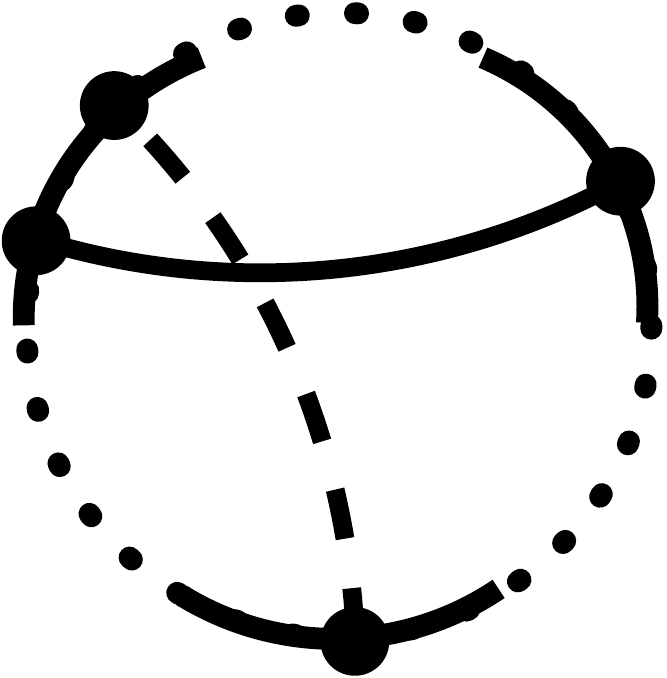}\raisebox{0.7cm}{=}
\includegraphics[width=1.5cm]{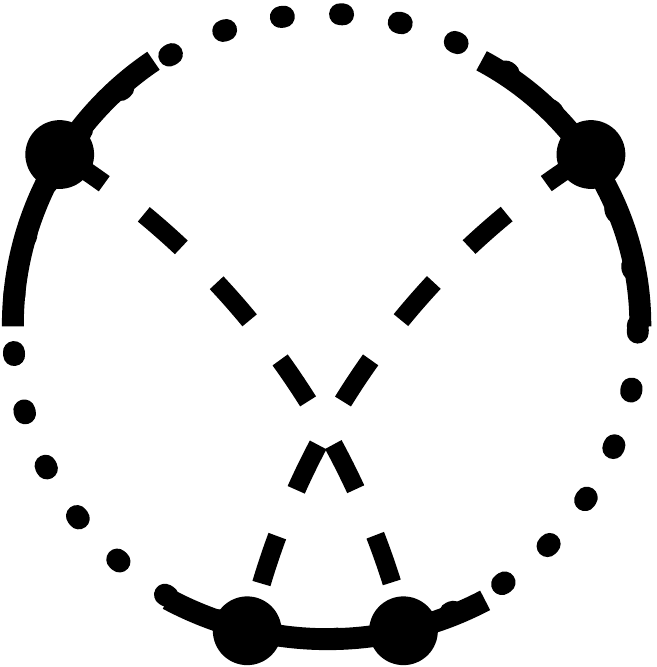}\raisebox{0.7cm}{--}
\includegraphics[width=1.5cm]{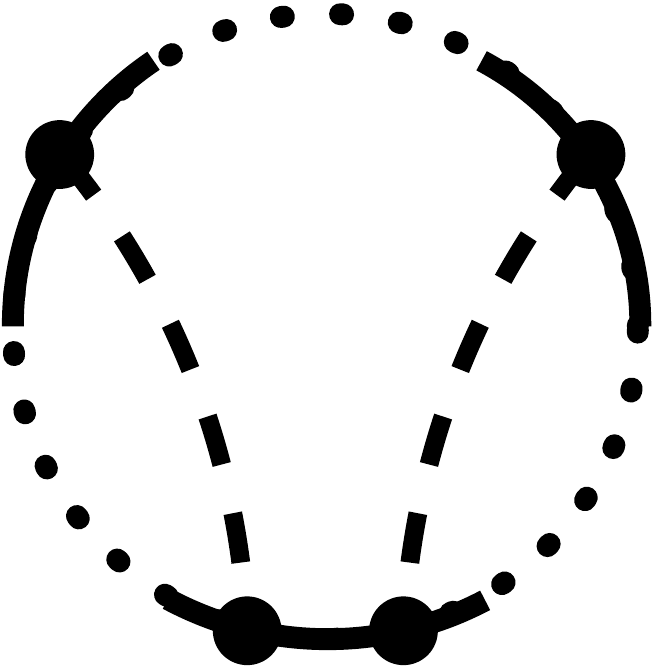}\raisebox{0.7cm}{=}
\includegraphics[width=1.5cm]{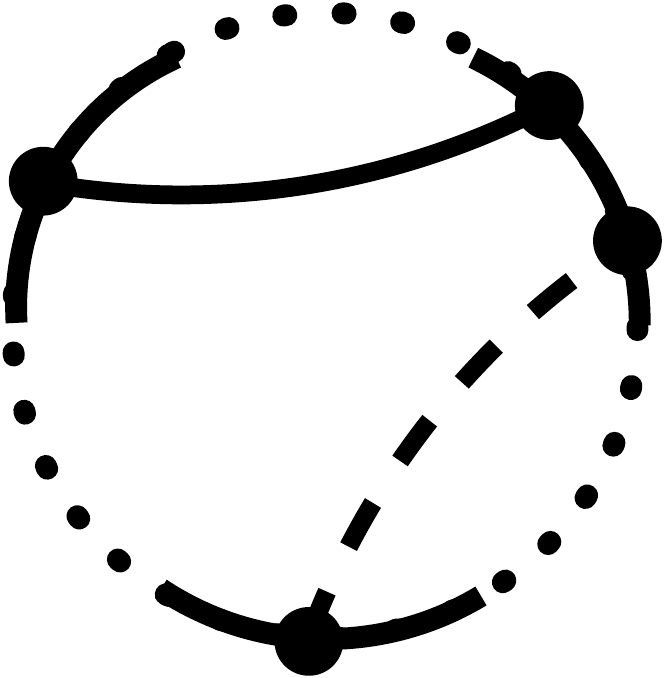}\raisebox{0.7cm}{--}
\includegraphics[width=1.5cm]{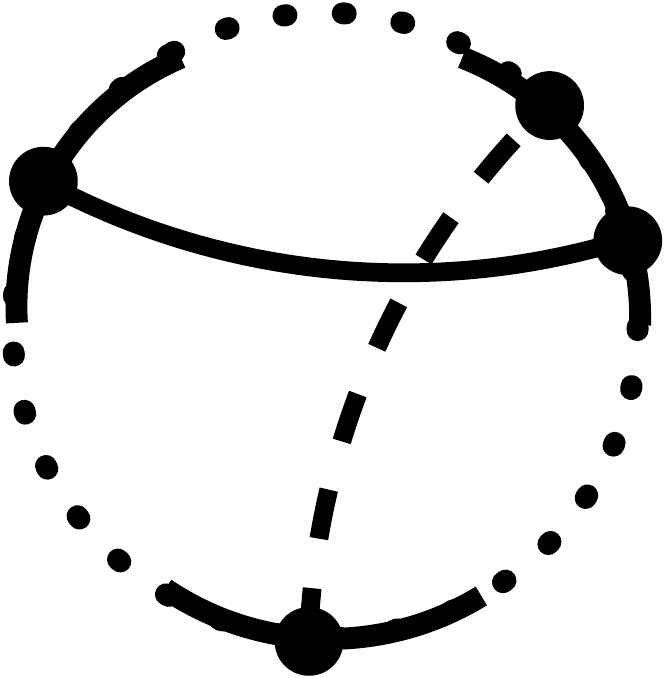}
(b)
\caption{4T-relations. The relations of type (a) are common for both $\mathcal A$ and $\mathcal M$.}\label{4T}
\end{center}
\end{figure}
We also refer to these relations as to \emph{4T-relations}.
\end{defn}

Clearly, the space $\mathcal M$ is naturally graded by the total number of chords.

Unfortunately, a natural way to define a multiplication on $\mathcal M$ is not known yet. However, the multiplication of a framed diagram by a black diagram is well-defined.

\begin{prop}
Let $c$ be a black diagram, and $d$ an arbitrary framed chord diagram. Then the product of $c$ and $d$, defined by cutting and glueing the two circles (see Figure~\ref{cd_prod}), is well-defined modulo 4T-relation, i.e. it does not depend on the choice of the points where the diagrams are cut.
\end{prop}

\begin{figure}[h]
\begin{center}
\includegraphics[width=3cm]{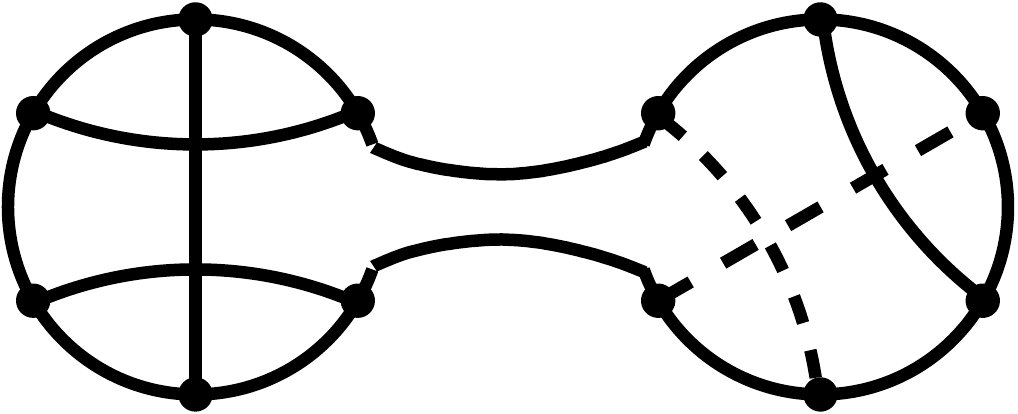} \raisebox{0.5cm}{=}
\includegraphics[width=1.2cm]{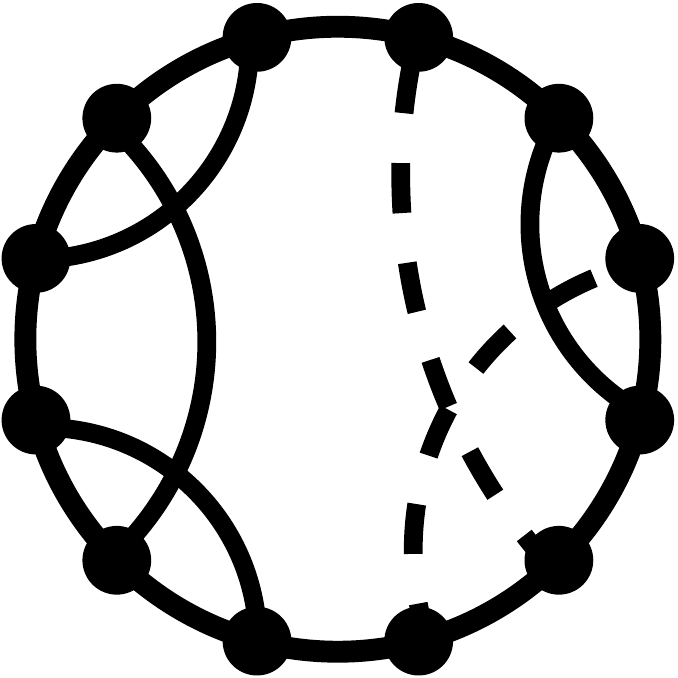}
\end{center}
\caption{Product of diagrams.}\label{cd_prod}
\end{figure}

\begin{proof}
Like in the non-framed diagrams case, this easily follows from the framed 4T-relations. It is enough to use only the relations that do not change the number of oriented chords. The proof of the fact that the product does not depend on the choice of the cut-point of $c$ repeats word-for-word the proof that the natural map from the algebra of  circular chord diagrams to the algebra of linear chord diagrams is well-defined (see~\cite{CDbook}, Section 4.7). The cut-point of $d$ choice independence follows by a direct generalization of the corresponding proof in the non-framed case (see~\cite{CDbook}, Lemma 4.13), since the framing of the jumping chord remains unchanged under a jump along an oriented chord.
\end{proof}

The bilinear extension of the above described product
$$m_\mathcal M\colon \mathcal A \otimes \mathcal M \to \mathcal M$$
makes $\mathcal M$ into a $\mathcal A$-module.

Now we endow $\mathcal M$ with an additional algebraic structure.

\begin{defn}
Let $H$ be a Hopf algebra over the field $\mathbb K$ with a coproduct $\Delta_H$, and a counit $\epsilon_H$. Then a $\mathbb K$-vector space $V$ is called a (left) $H$-comodule if there is a map
$$\delta_V \colon V\to H\otimes V,$$
satisfying the identities
$$ (id\otimes \delta_V)\circ \delta_V = (\Delta_H \otimes id)\circ \delta_V, \mbox{ and}$$
$$\iota \circ (\epsilon_H \otimes id) \circ \delta_V = id,$$
where $\iota$ is the canonical isomorphism $\mathbb K\otimes V \to V$.
\end{defn}

Recall (see~\cite{Lando}) that the space $\mathcal M$ admits a well-defined coalgebra structure given by the map $\Delta_\mathcal M$:
$$\Delta_\mathcal M \colon d\mapsto \sum d_I\otimes d_J,$$
where the summation is carried out over all possible ways to decompose the set of chords of $d$ into two non-intersecting subsets $I$ and $J$, and $d_I$ denotes the subdiagram of $d$ formed by the chords belonging to $I$.

There is an inclusion $\mathcal A\to \mathcal M$. Define the linear map $Pr_\mathcal M\colon \mathcal M \to \mathcal A$  as
$$Pr_\mathcal M (d) = \left\{\begin{array}{l} d,\mbox{ if $d$ is a black diagram,}\\									
						0, \mbox{ otherwise.}\end{array}\right.$$

The well-definedness of $Pr_\mathcal M$ is clear.

Put
$$\delta_\mathcal M = (Pr_\mathcal M\otimes id)\circ \Delta_\mathcal M \colon \mathcal M \to \mathcal A\otimes \mathcal M.$$
As it is just a composition of linear maps, it is also a well-defined linear map.

\begin{prop}
The map $\delta_\mathcal M$ endows $\mathcal M$ with an $\mathcal A$-comodule structure.
\end{prop}
\begin{proof}
The identities $(id\otimes \delta_\mathcal M)\circ \delta_\mathcal M = (\Delta_\mathcal A \otimes id)\circ \delta_\mathcal M$ and $ \iota \circ (\epsilon_\mathcal A \otimes id) \circ \delta_\mathcal M = id$ are straightforward. The first one follows from the presentation of the both sides of the equality as the sums over the decompositions of the set of oriented chords of a diagram into three non-intersecting sets. The second one follows from the fact that the only diagram in $\mathcal A$ having a non-vanishing image under $\epsilon_\mathcal A$ is the empty diagram.
\end{proof}

The module and the comodule structures on $\mathcal M$ are tightly connected.

\begin{defn}
Let $H$ be a Hopf algebra over $\mathbb K$, with a product $m_H$ and a coproduct $\Delta_H$, and let $V$ be a $\mathbb K$-vector space. Then $V$ is called a (left) \emph{$H$-Hopf module} if
\begin{itemize}
\item There is a map $m_V \colon H\otimes V \to V$ endowing $V$ with a (left) $H$-module structure;
\item There is a map $\delta_V \colon V \to H\otimes V$ endowing $V$ with a (left) $H$-comodule structure;
\item The map $\delta_V$ is a morphism of $H$-modules,  where the $H$-module structure on $H\otimes V$ is given by the map
\begin{multline*}
m_{H\otimes V} = (m_H\otimes m_V)\circ (id\otimes s \otimes id)\circ(\Delta_H \otimes (id\otimes id)) \colon \\ H\otimes (H\otimes V) \to H\otimes V,
\end{multline*}
where $s\colon H\otimes H \to H\otimes H$ is the map permuting the tensor factors.
\end{itemize}
\end{defn}

\begin{thm}
The maps $m_\mathcal M$ and $\delta_\mathcal M$ endow $\mathcal M$ with a structure of  an $\mathcal A$-Hopf module.
\end{thm}
\begin{proof}
We have to check that the coproduct map $\delta_\mathcal M$ is a morphism of $\mathcal A$-modules $\mathcal M$ and $\mathcal A\otimes \mathcal M$. Indeed, consider a product $g = m_{\mathcal M} (c\otimes d)$, where $c$ is a black diagram, and $d$ is an arbitrary framed chord diagram.  By definition,
$$\delta_\mathcal M (cd) = \sum_{I \subset OC_{g}} g_I \otimes g'_I,$$
where $OC_{g}$ denotes the set of oriented chords of the framed chord diagram $g=cd$, $g_I$ denotes the subdiagram consisting of the chords of the set $I$ only (with the chords combinatorial structure preserved), and $g'_I$ is the diagram with the chords from the set $I$ \emph{removed}.

Note that every subset $I\subset OC_{g}$ can be presented as a disjoint union $I_c \cup I_d \subset OC_c\cup OC_d$. 
Hence,
$$ \delta_\mathcal M (cd) = \sum_{I_c\subset OC_c} \sum_{I_d\subset OC_d} g_{I_c\cup I_d} \otimes g'_{I_c\cup I_d} =\sum_{I_c\subset OC_c} \sum_{I_d\subset OC_d} c_{I_c}d_{I_d} \otimes c'_{I_c}d'_{I_d},$$
which clearly equals to $m_{\mathcal A\otimes \mathcal M}(c\otimes\delta_\mathcal M(d))$.
\end{proof}

The main structure result of the theory of chord diagrams is a direct corollary of the Milnor-Moore theorem. It states that every commutative cocommutative connected graded Hopf algebra over a field of characteristic 0 is a polynomial algebra generated by a basis of the subspace of its primitive elements. There is a similar structure theorem for Hopf modules. To formulate it, we need one more definition.

\begin{defn}
Let $H$ be a Hopf algebra over $\mathbb K$, and let $V$ be a (left) $H$-Hopf module with the comodule structure map $\delta_V$. The \emph{covariant submodule} $Co_V\subset V$ is the subspace of $V$ spanned by the covariant elements, i.e., elements $v\in V$ 
satisfying $\delta_V(v) = 1\otimes v$.
\end{defn}

\begin{thm}[Larson-Sweedler, \cite{LS}]
Let $H$ be a commutative cocommutative Hopf algebra over $\mathbb K$ with the coproduct $\Delta_H$, and let $V$ be an $H$-Hopf module with the module structure given by the product $m_V$. Then $V$ is trivial, i.e. the restriction of the product map
$$m_V |_{H\otimes Co_V}\colon H\otimes Co_V \to V$$
is an $H$-Hopf modules isomorphism. The $H$-Hopf module structure on $H\otimes Co_V$ is given by the maps
$$m_{H\otimes Co_V}\colon h_1\otimes (h_2\otimes v) \mapsto h_1 h_2\otimes v, \mbox{ for any $h_1,h_2\in H, v\in Co_V$, and}$$
$$\Delta_{H\otimes Co_V} (h\otimes v) \mapsto \Delta_H(h)\otimes v \mbox{ for any $h\in H, v\in Co_V$}.$$
\end{thm}

An immediate corollary of the Larson-Sweedler theorem is

\begin{cor}
The $\mathcal A$-Hopf module $\mathcal M$ is free. It is generated by a basis of its covariant submodule $Co_\mathcal M$.
\end{cor}

 This statement means that every element of $\mathcal M$ is representable as a linear combination of products of the type $cd$, where $c$ is a black diagram, and $d$ is a covariant element.
The space  $Co_\mathcal M$ admits a nice description.

\begin{prop}\label{structure}
The space $Co_\mathcal M$ is spanned by white diagrams.
\end{prop}
\begin{proof}
It is clear that the subspace spanned by white diagrams is a subspace of $Co_\mathcal M$.

Now we are going to show that every chord diagram modulo 4T-relations is representable as a linear combination 
of products of black diagrams by white diagrams.

Take an arbitrary framed chord diagram $d$. If all its chords have the same framing, then there is nothing to prove.

Now suppose that the diagram contains chords of both possible framings. Choose a point $p$ different from the chords' endpoints on the supporting circle. Let $h$ be an endpoint of an oriented chord. By the \emph{remoteness} $r_p(h)$ of $h$ with respect to $p$ we mean the number of endpoints of disoriented chords one meets while traveling along the circle with respect to its orientation from $h$ to $p$. The number
$$ C_p(d) = 5^b \frac{(b-1)!}{w!} \sum_h r_p(h),$$
where $b$ and $w$ denote the number of oriented and disorienting chords, respectively, is called the \emph{complexity} of $d$ with respect to $p$.

Note that if the complexity equals 0, then $d$ obviously is representable as a product of a black and a white diagram. The minimal possible non-zero complexity of a diagram with $b$ oriented chords and $w$ disorienting chords is $5^b\frac{(b-1)!}{w!}$.

Travelling along the circle starting from $p$ in a clockwise direction, consider the first encountered oriented chord endpoint with a non-zero remoteness. Apply the 4T-relation, as shown in Figure~\ref{4T}~(b), trying to reduce the remoteness of the oriented chord endpoint under consideration. In such a way we present $d$ as a linear combination of three diagrams, with the basis point $p$ inherited to all three. One of these diagrams differs from $d$ by the position of one of the oriented chord endpoints only, and its complexity is clearly smaller than that of $d$.

The number of oriented chords in each of the two other diagrams is one less than in $d$. If they do not contain oriented chords at all, then we are done. Otherwise, we take into account that the largest possible complexity of a diagram with 
$b-1$ oriented chords and $w+1$ disorienting chords is
$$ 4(b-1)(w+1)\cdot 5^{b-1}\frac{(b-2)!}{(w+1)!} = 4\cdot 5^{b-1}\frac{(b-1)!}{w!}$$
which is evidently smaller than  $5^b\frac{(b-1)!}{w!}$.
\end{proof}

Another corollary of the Larson--Sweedler theorem is
\begin{cor}
The space $\mathcal M$ admits a bigrading.
\end{cor}
\begin{proof}
Indeed, by the Larson-Sweedler theorem, $\mathcal M$ is a tensor product of the graded algebra $\mathcal A$ and the vector space spanned by white diagrams. The latter is also graded since the framed 4T-relations are homogeneous with respect to the number of chords. The assertion follows.
\end{proof}

\emph{Remark:} Proposition~\ref{structure} means that the symbol of any finite-type J-invariant is a linear combination of products of symbols of  J$^+$-invariants by symbols of J$^-$-invariants.\vspace{0.3cm}

The space $Co_\mathcal M$ has non-trivial relations among the generating white diagrams. Up to degree 4, there are no relations, and the first example of a non-trivial relation in degree 4 was found by S.~Duzhin. Examples of relations in $Co_\mathcal M$ are shown in Figure~\ref{white}. Unfortunately, a complete description of the set of relations in $Co_\mathcal M$ is not known.

\begin{figure}[h]
\begin{center}
\includegraphics[width=1cm]{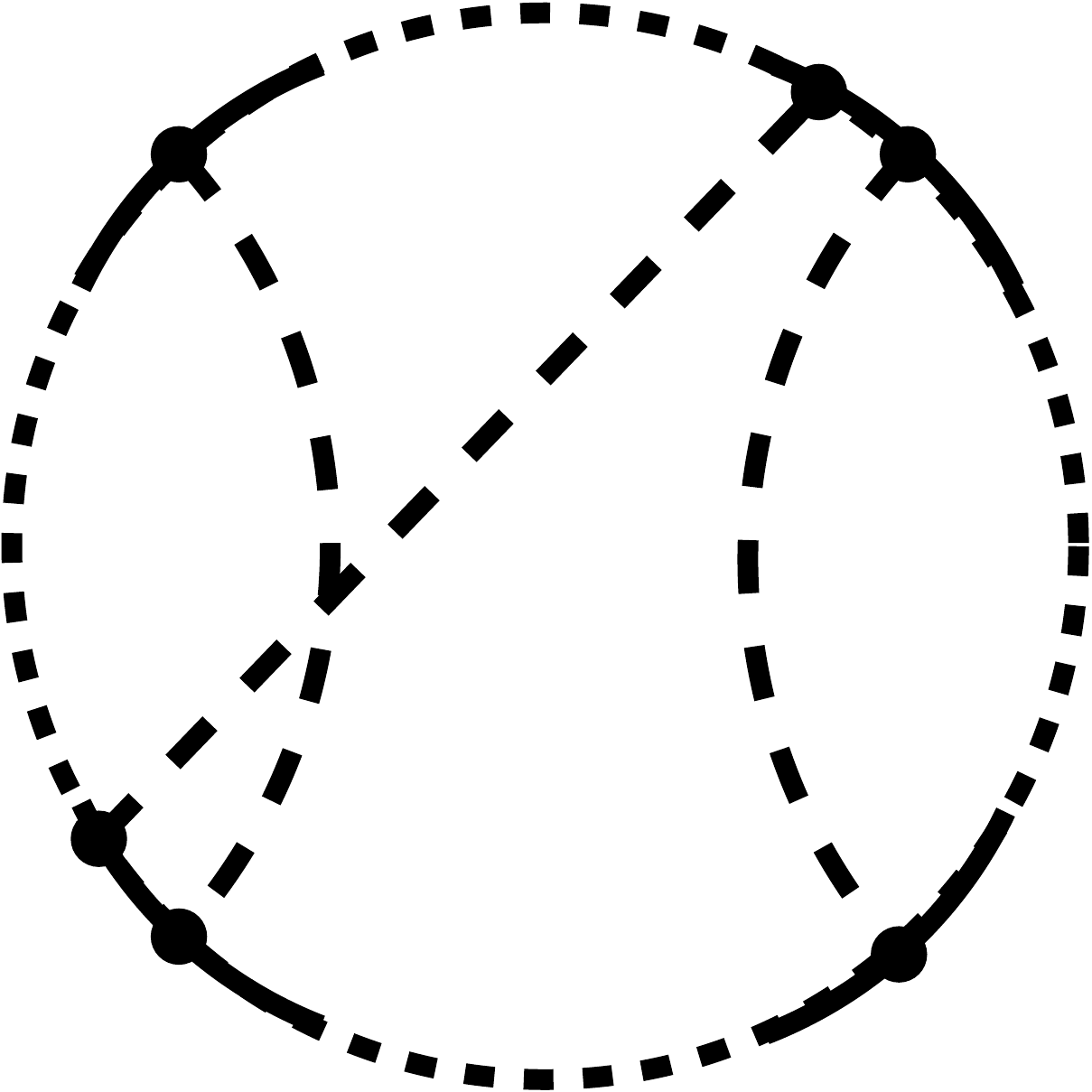} \raisebox{0.4cm}{--} \includegraphics[width=1cm]{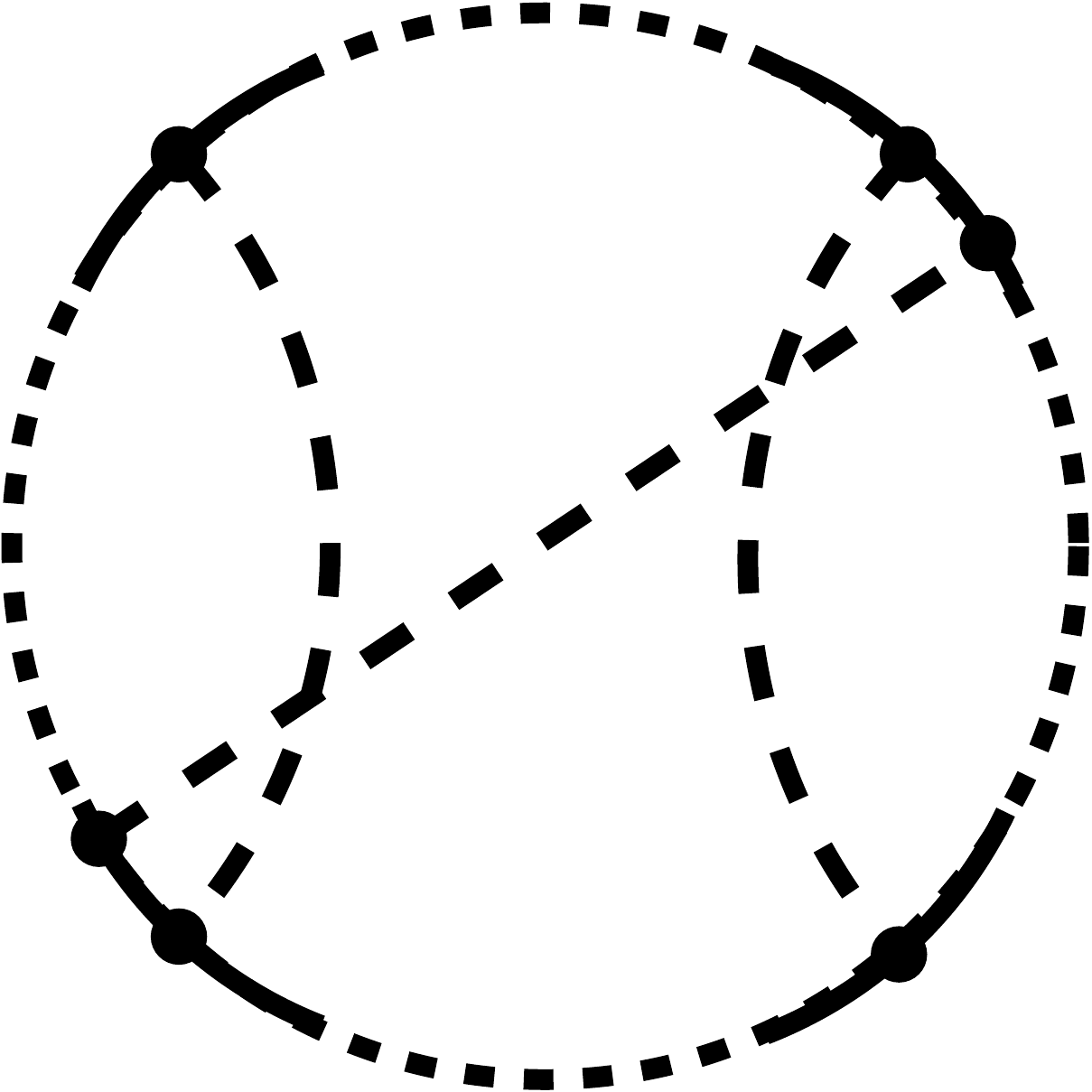} \raisebox{0.4cm}{+} \includegraphics[width=1cm]{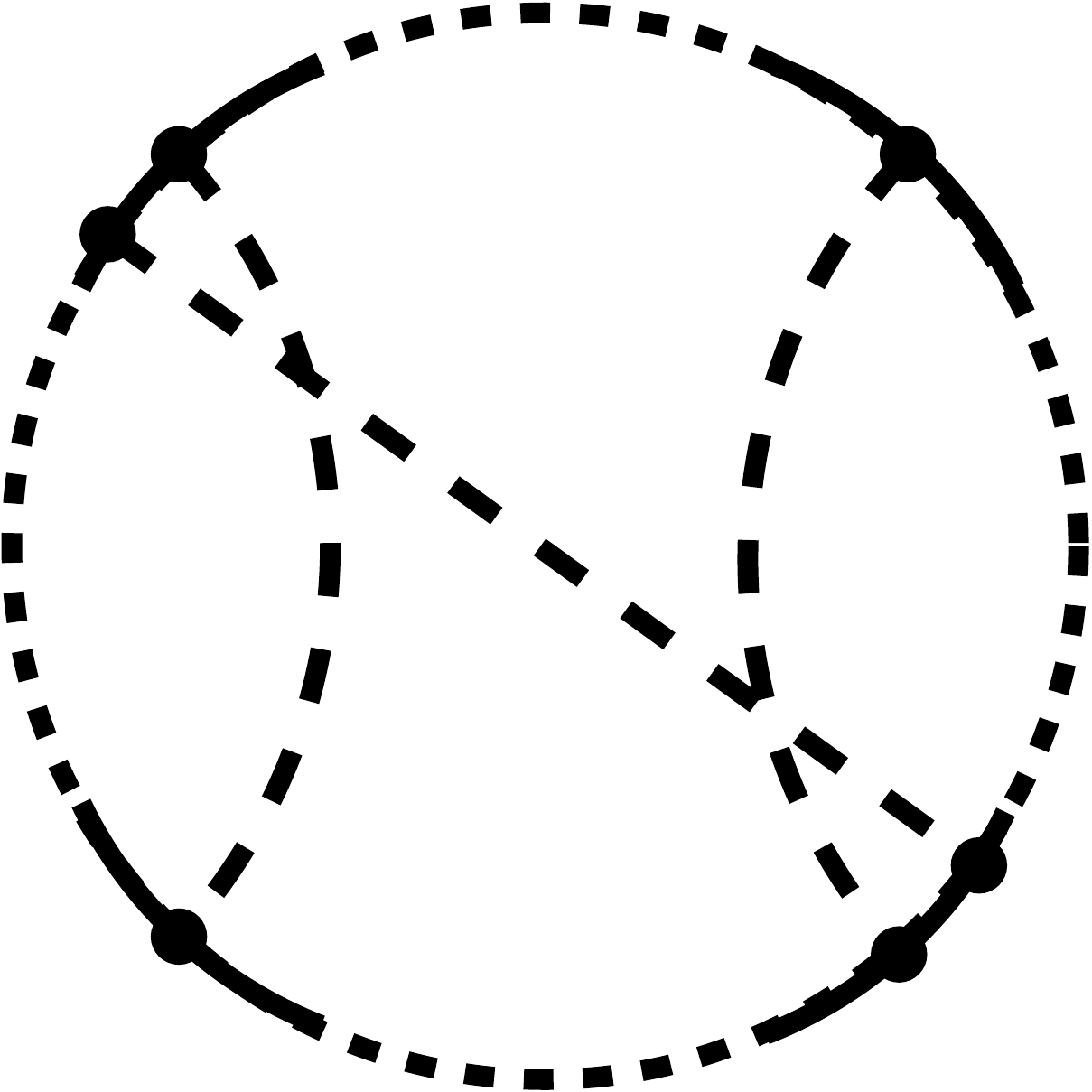} \raisebox{0.4cm}{--} \includegraphics[width=1cm]{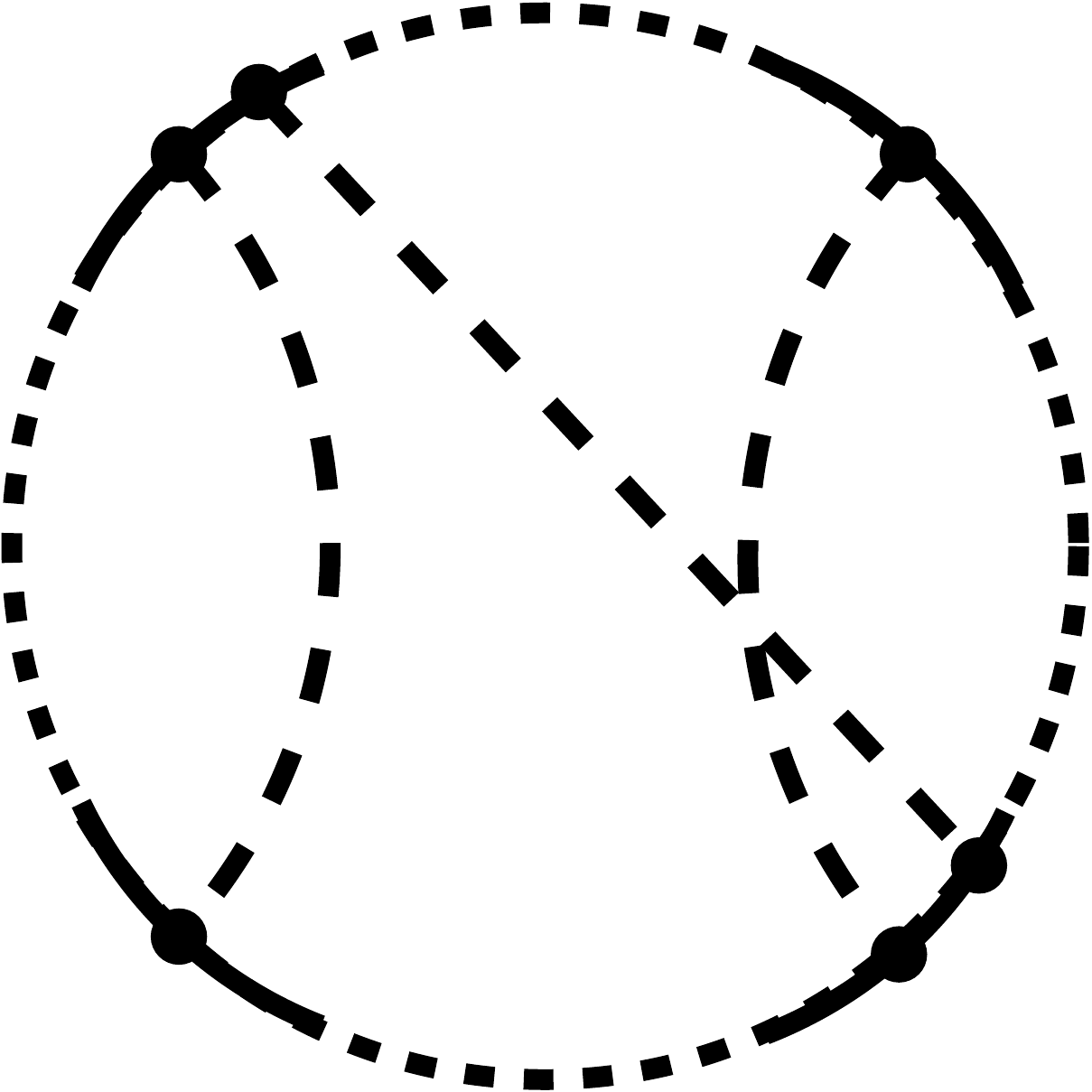} \raisebox{0.4cm}{+} \includegraphics[width=1cm]{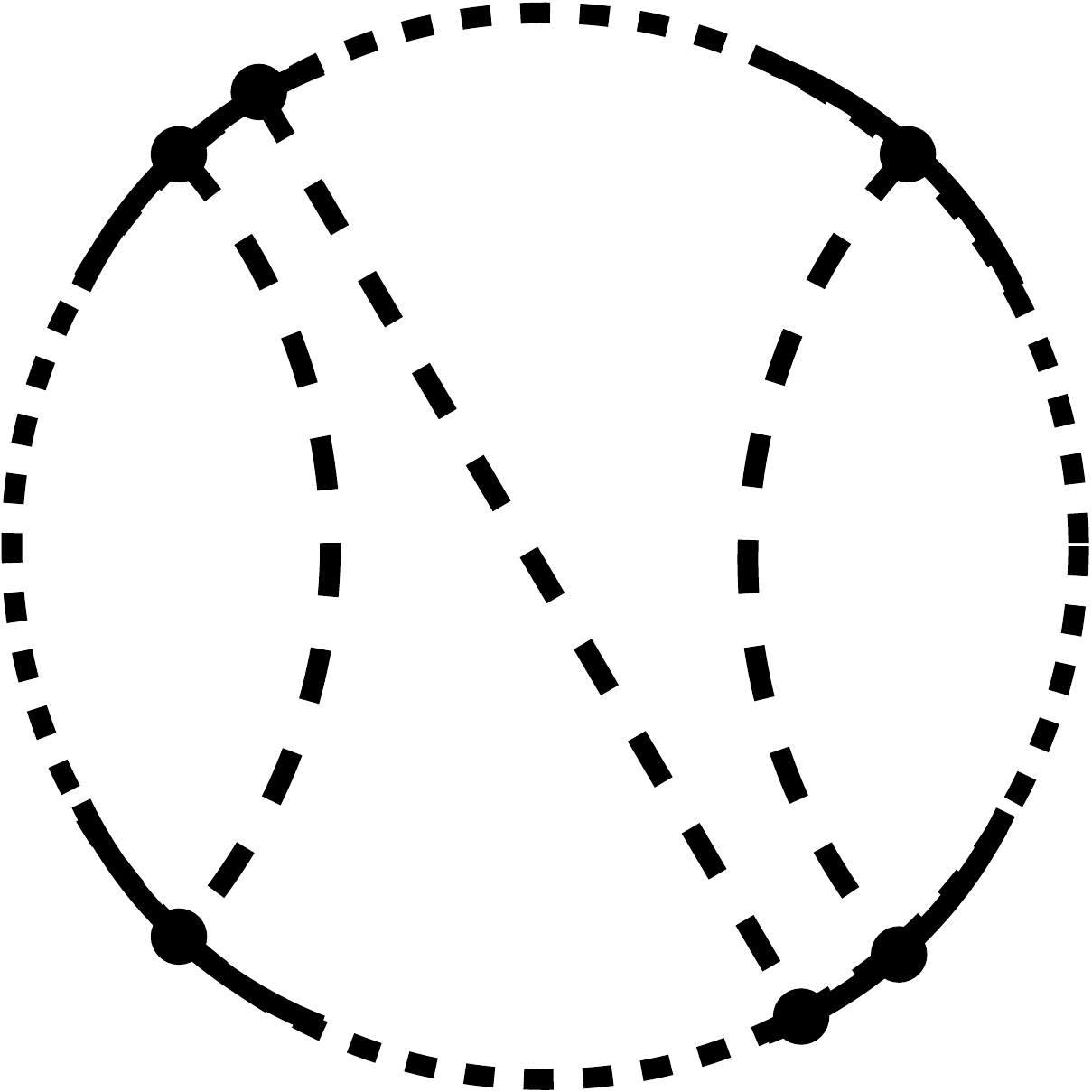}
\raisebox{0.4cm}{--} \includegraphics[width=1cm]{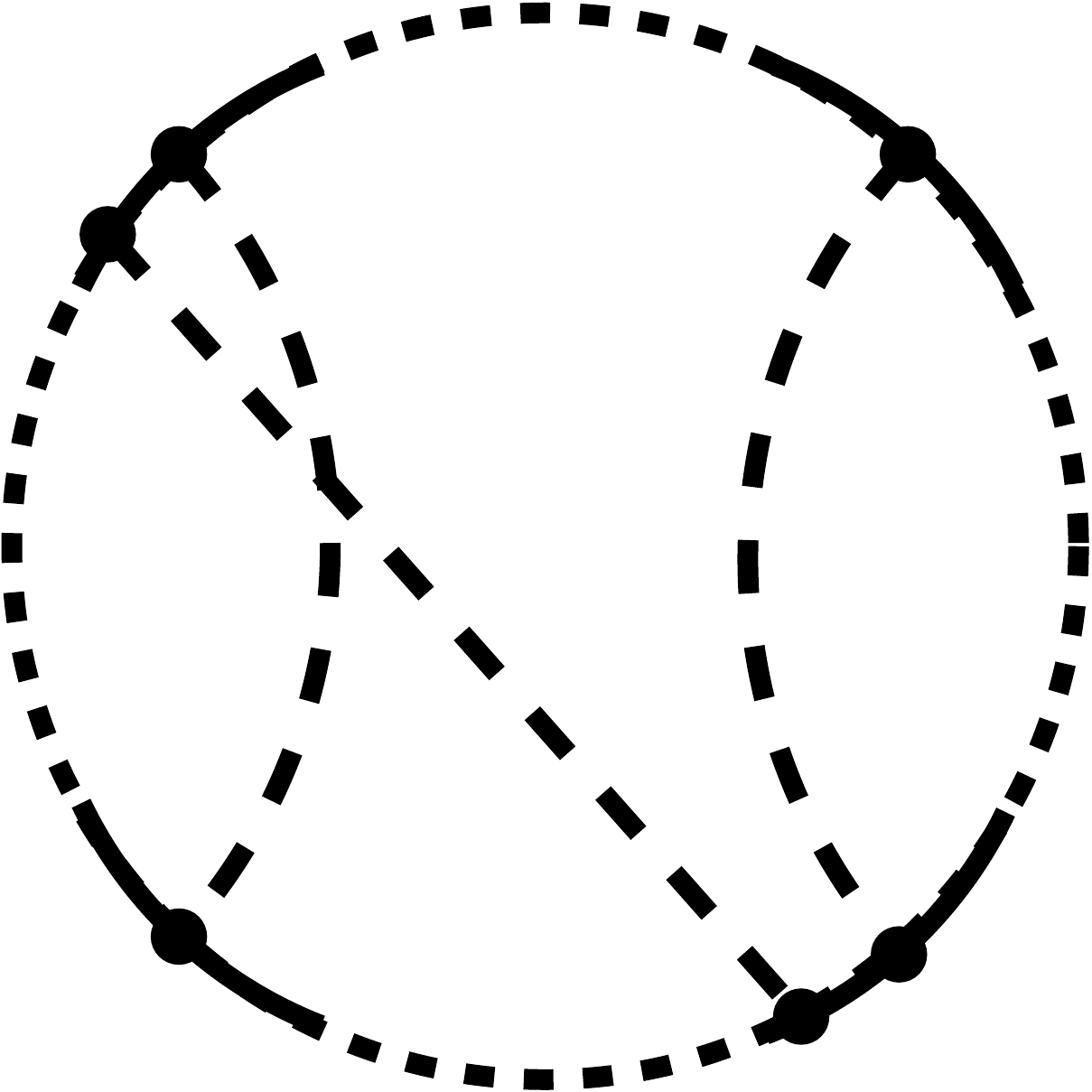}
\raisebox{0.4cm}{+} \includegraphics[width=1cm]{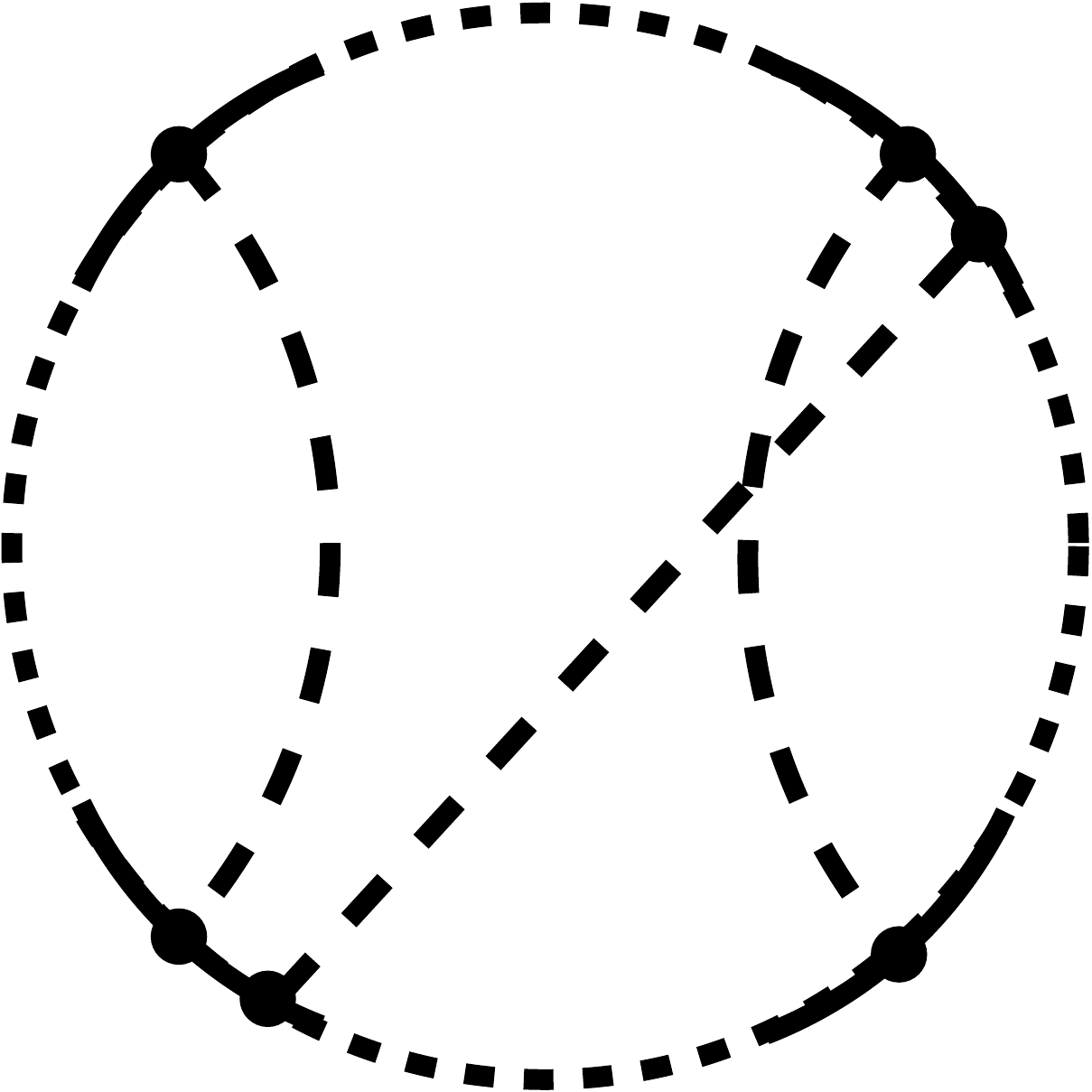}
\raisebox{0.4cm}{--} \includegraphics[width=1cm]{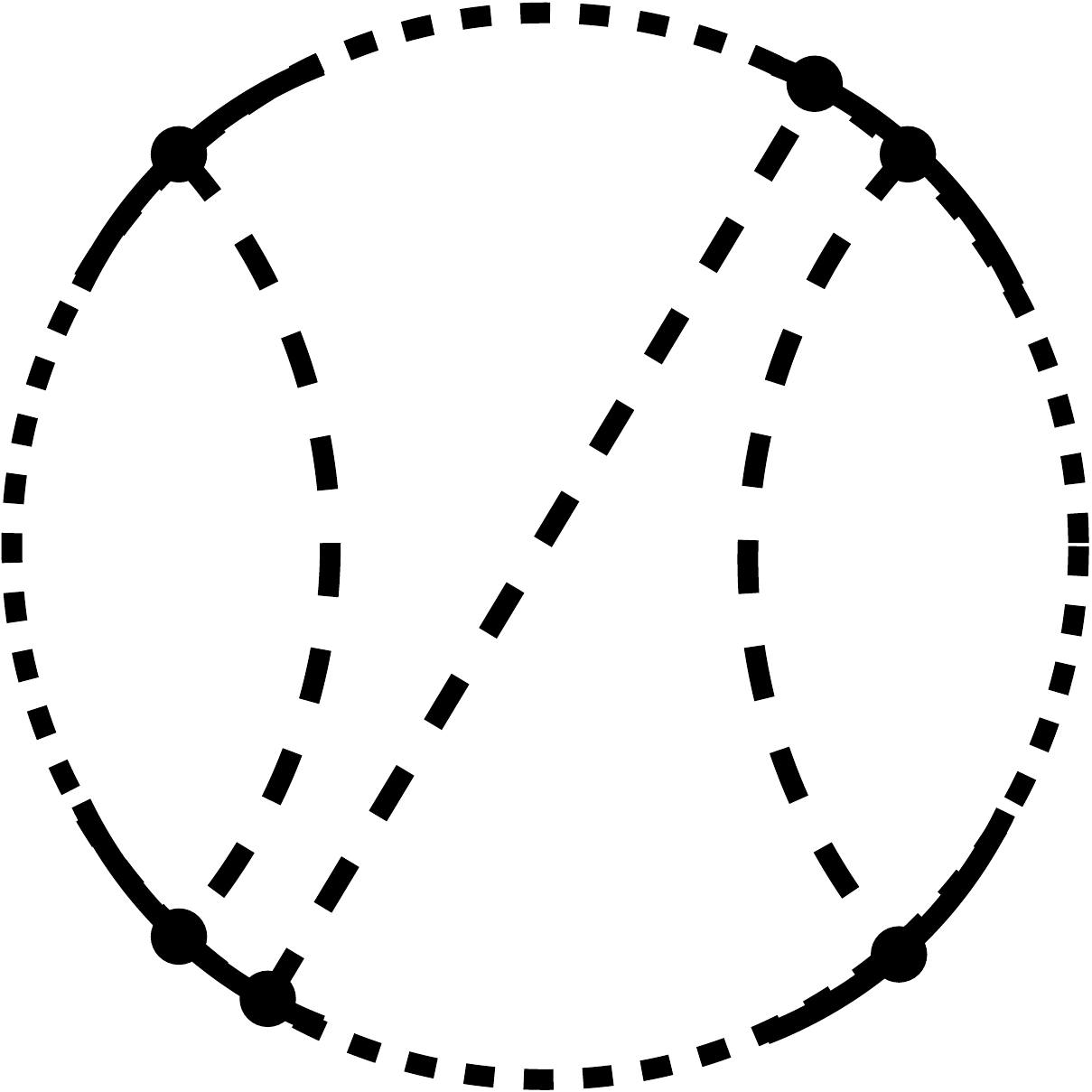}
 \raisebox{0.4cm}{=0}
\vspace{0.4cm}

 \includegraphics[width=1cm]{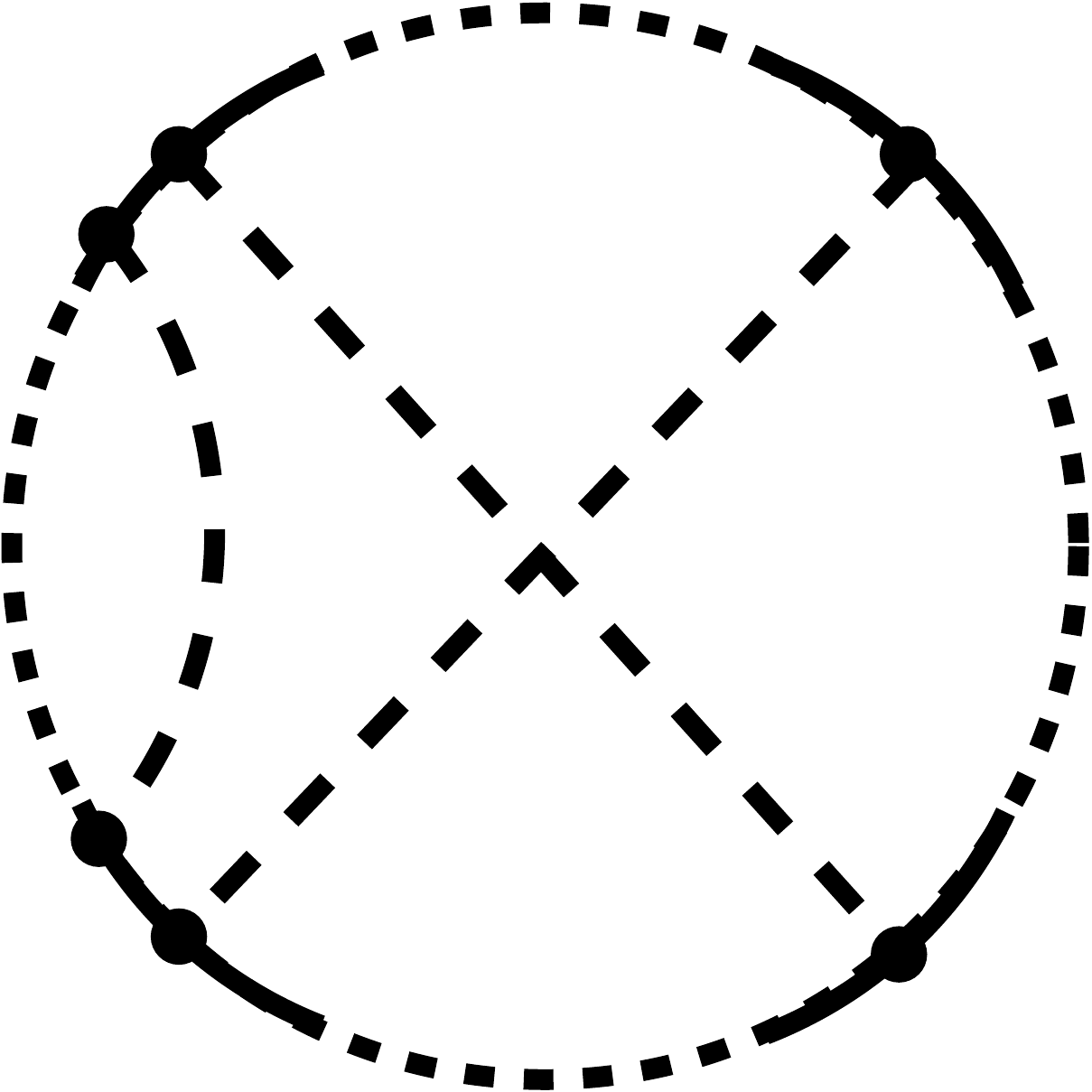} \raisebox{0.4cm}{--} \includegraphics[width=1cm]{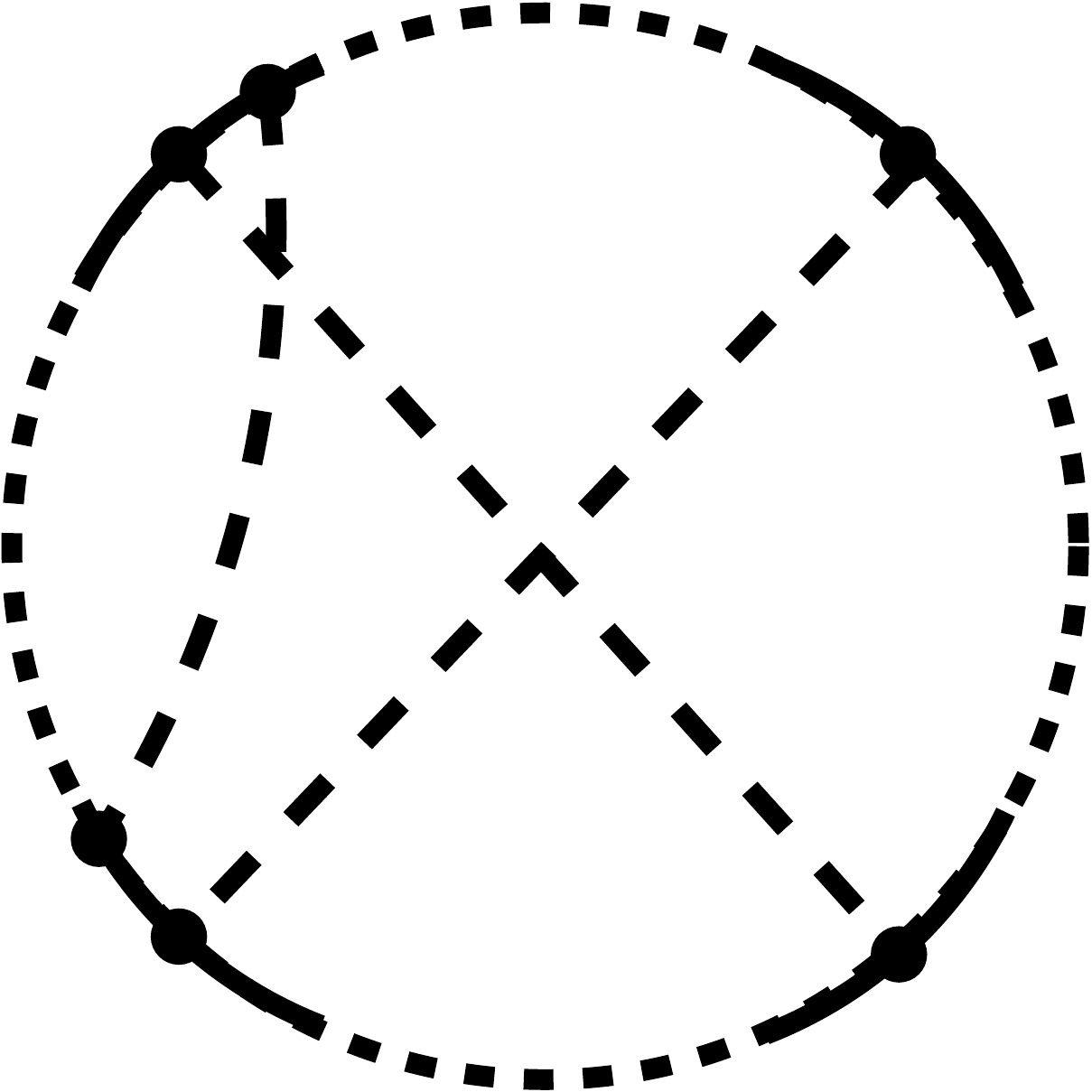} \raisebox{0.4cm}{+} \includegraphics[width=1cm]{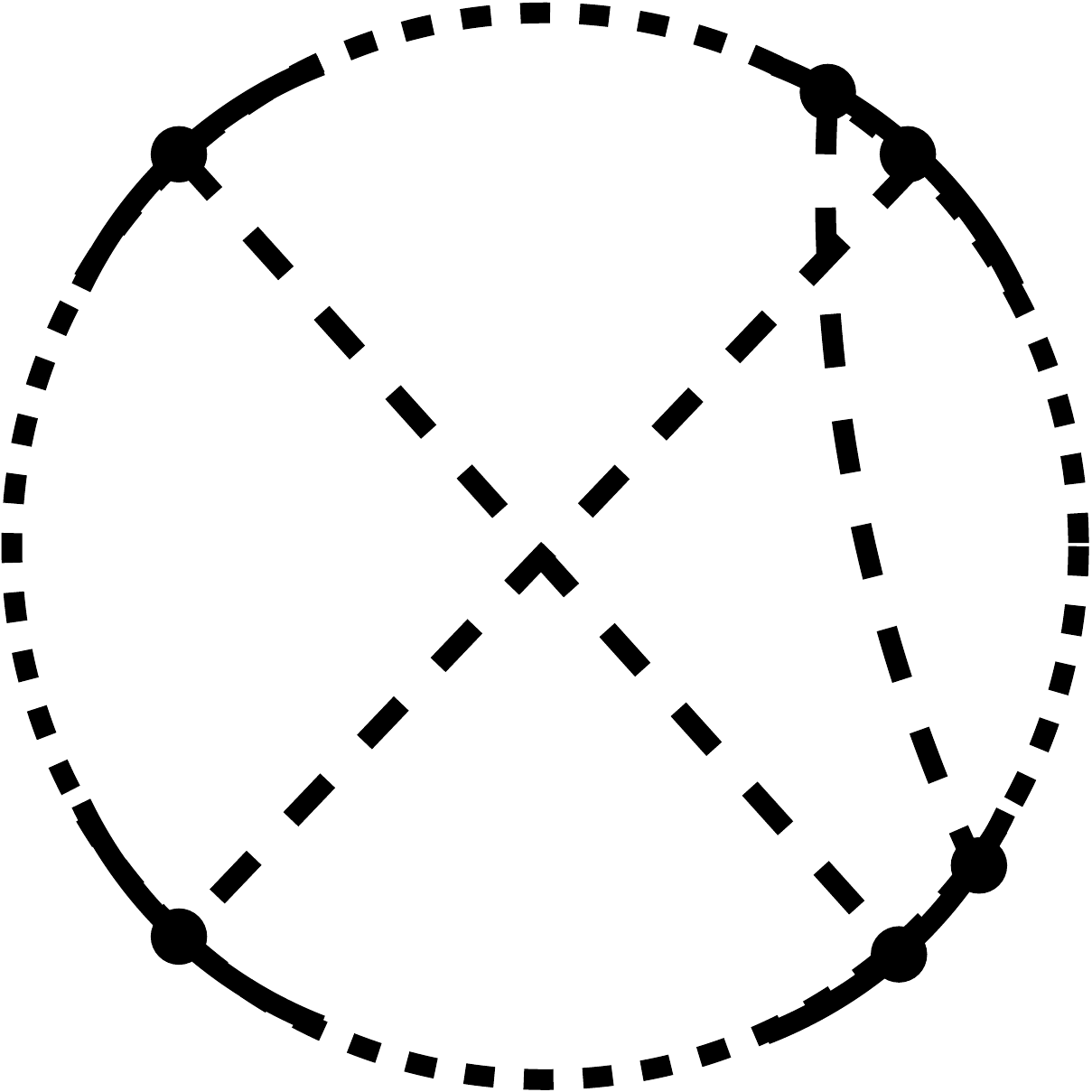} \raisebox{0.4cm}{--} \includegraphics[width=1cm]{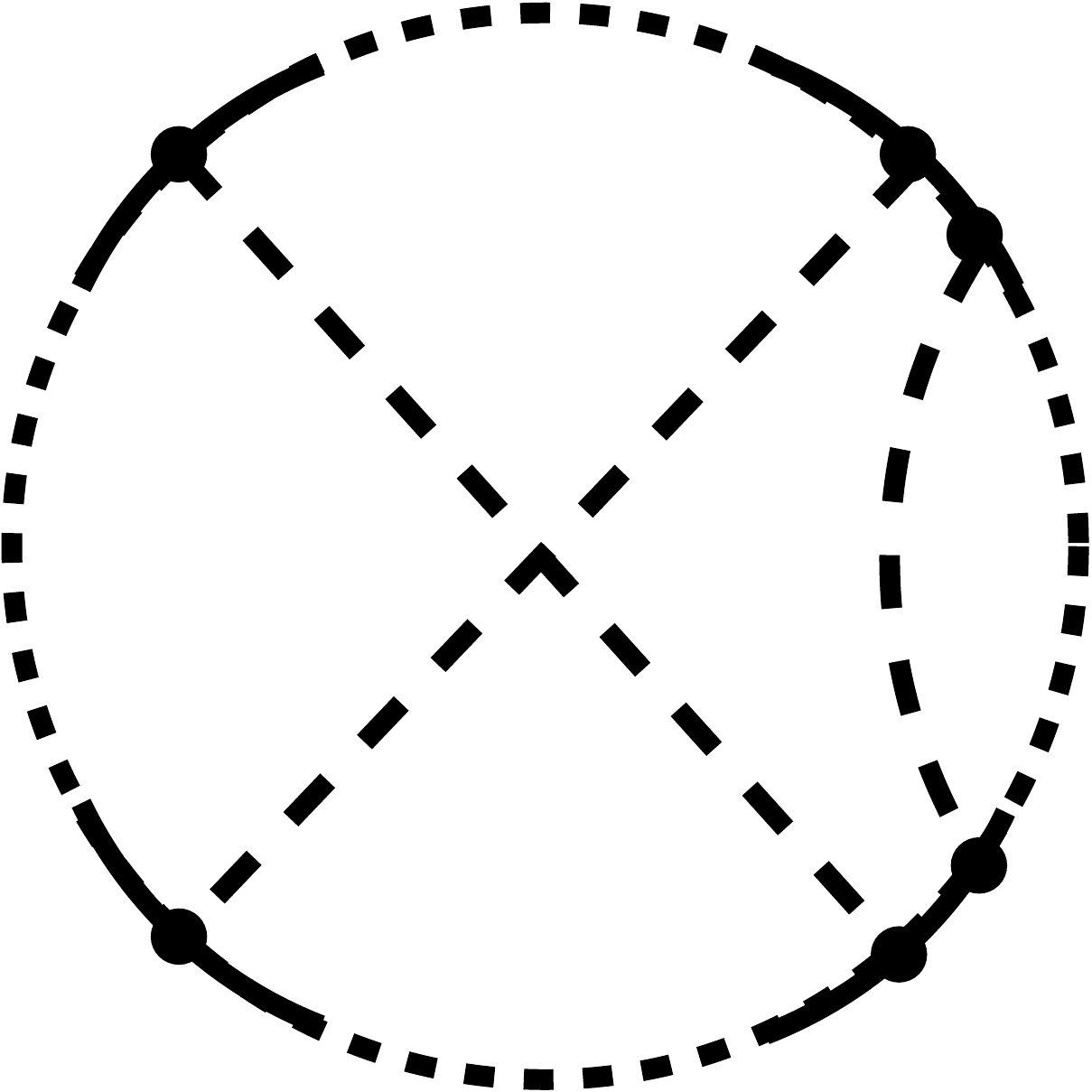} \raisebox{0.4cm}{+} \includegraphics[width=1cm]{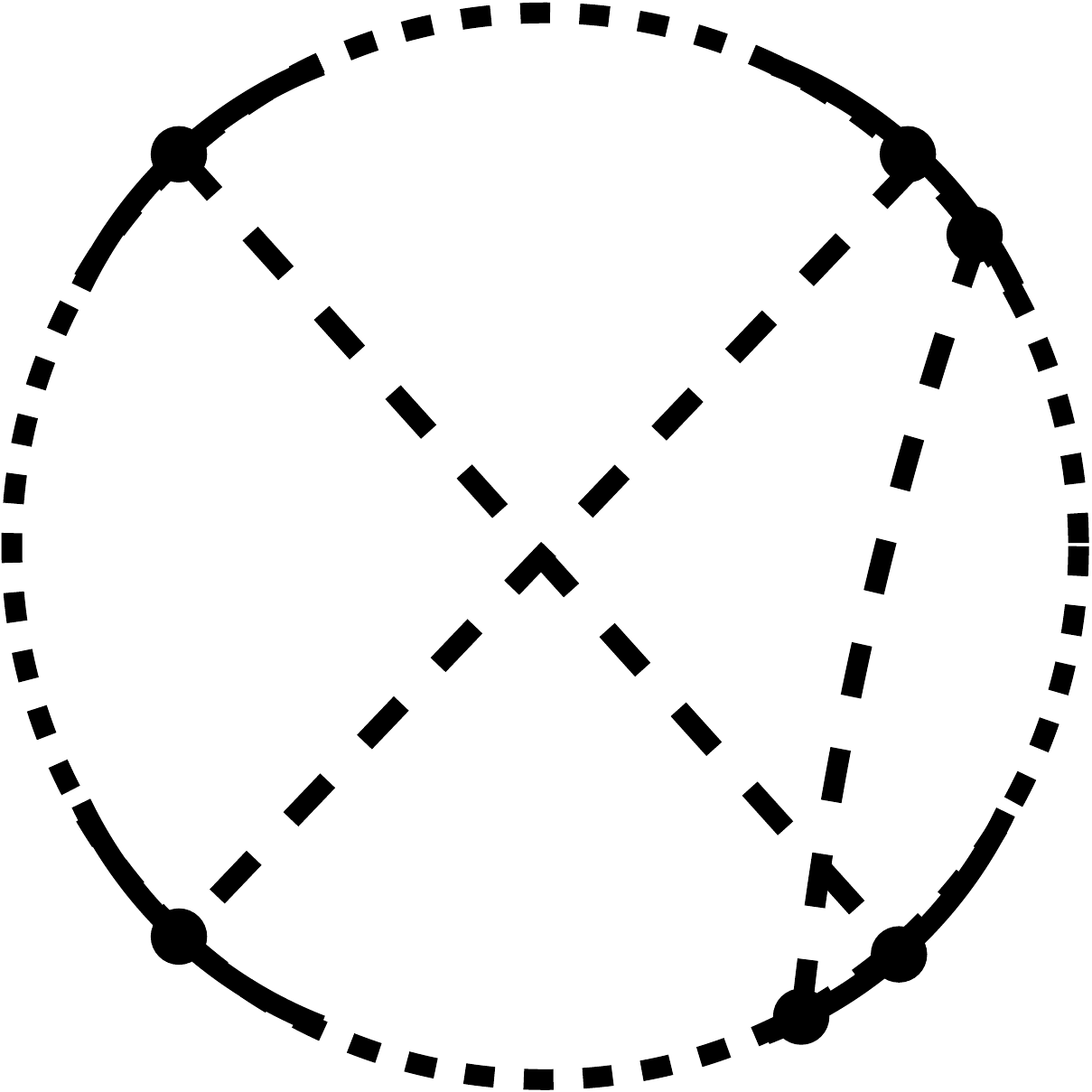}
\raisebox{0.4cm}{--} \includegraphics[width=1cm]{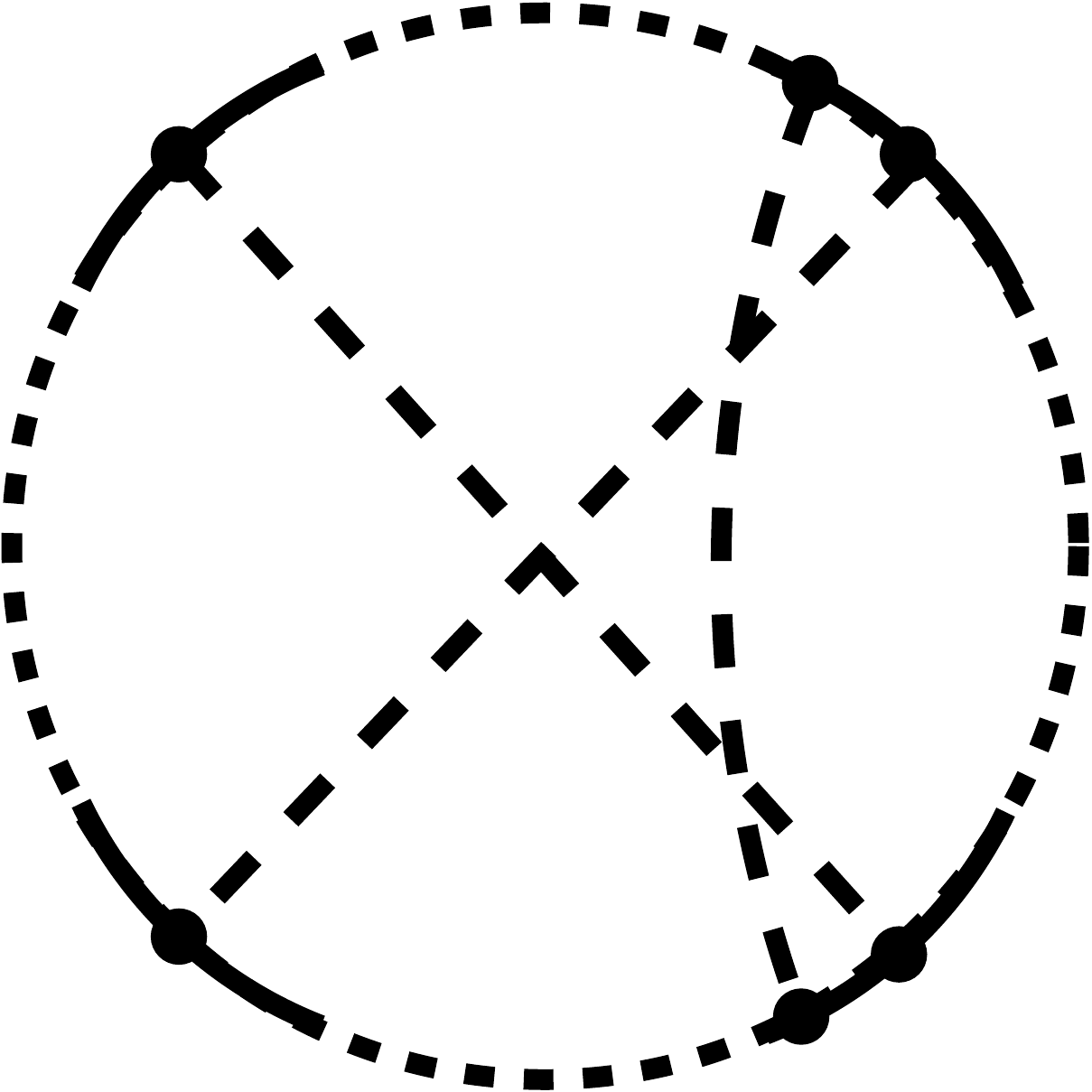}
\raisebox{0.4cm}{+} \includegraphics[width=1cm]{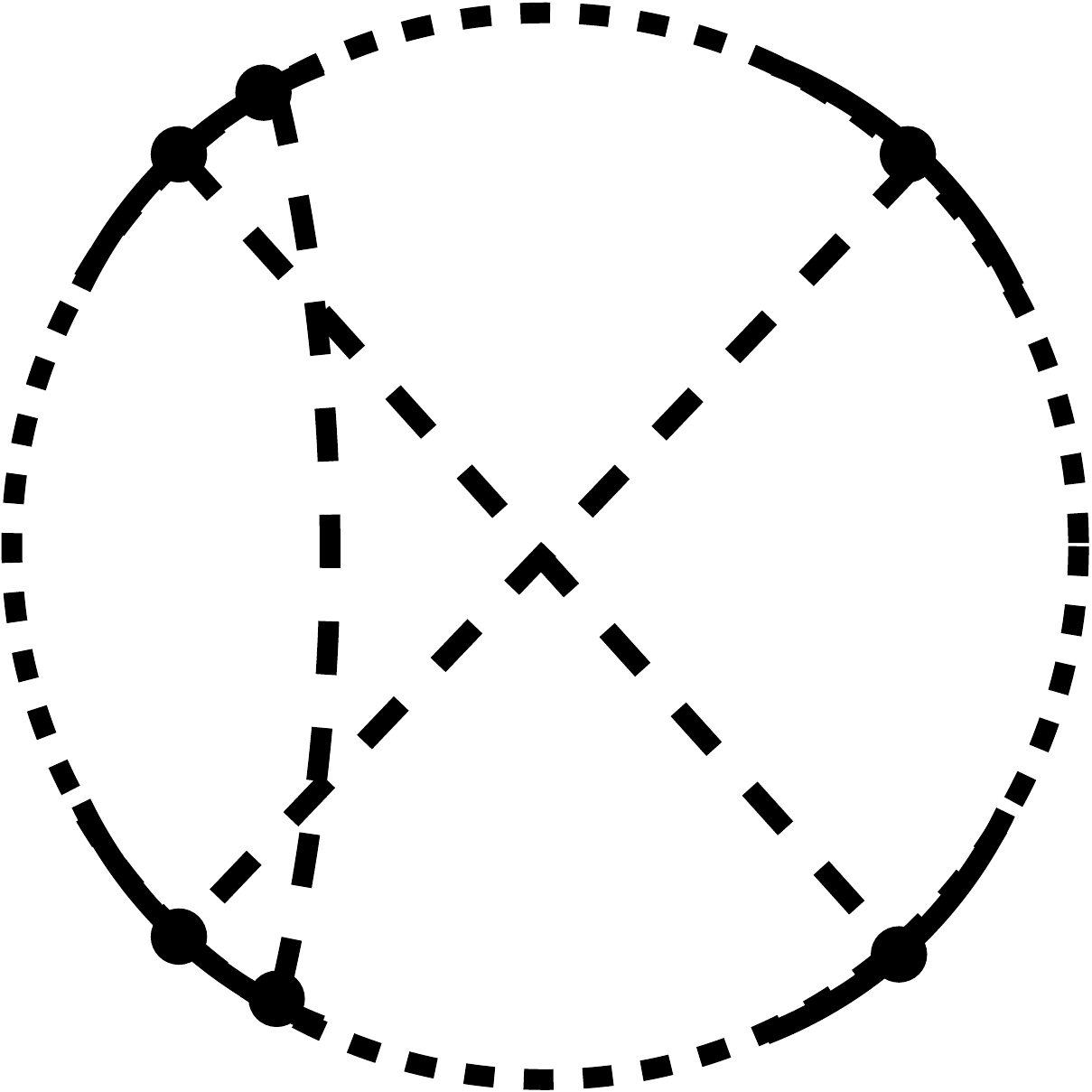}
\raisebox{0.4cm}{--} \includegraphics[width=1cm]{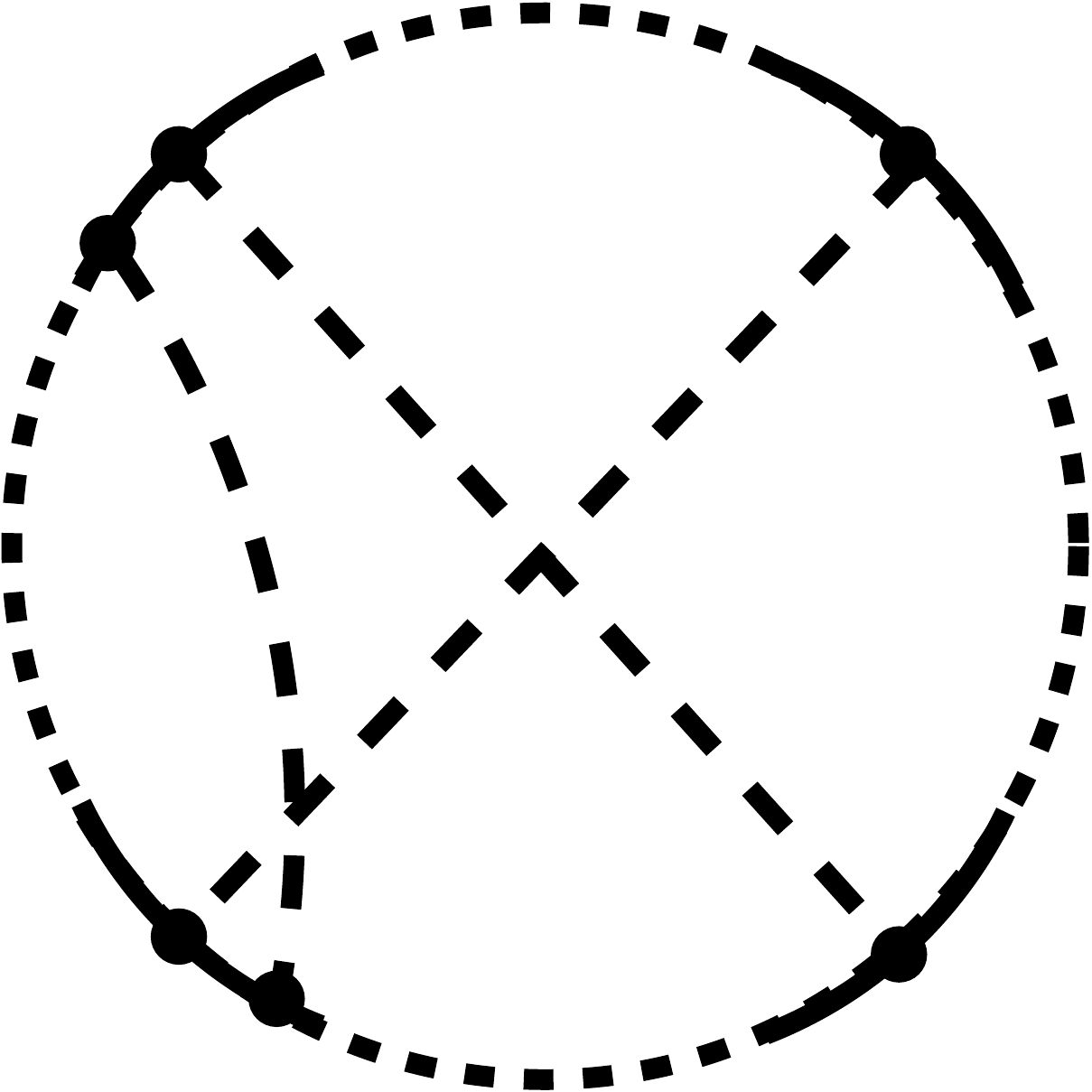}
 \raisebox{0.4cm}{=0}
\vspace{0.4cm}

\raisebox{0.4cm}{2} \includegraphics[width=1cm]{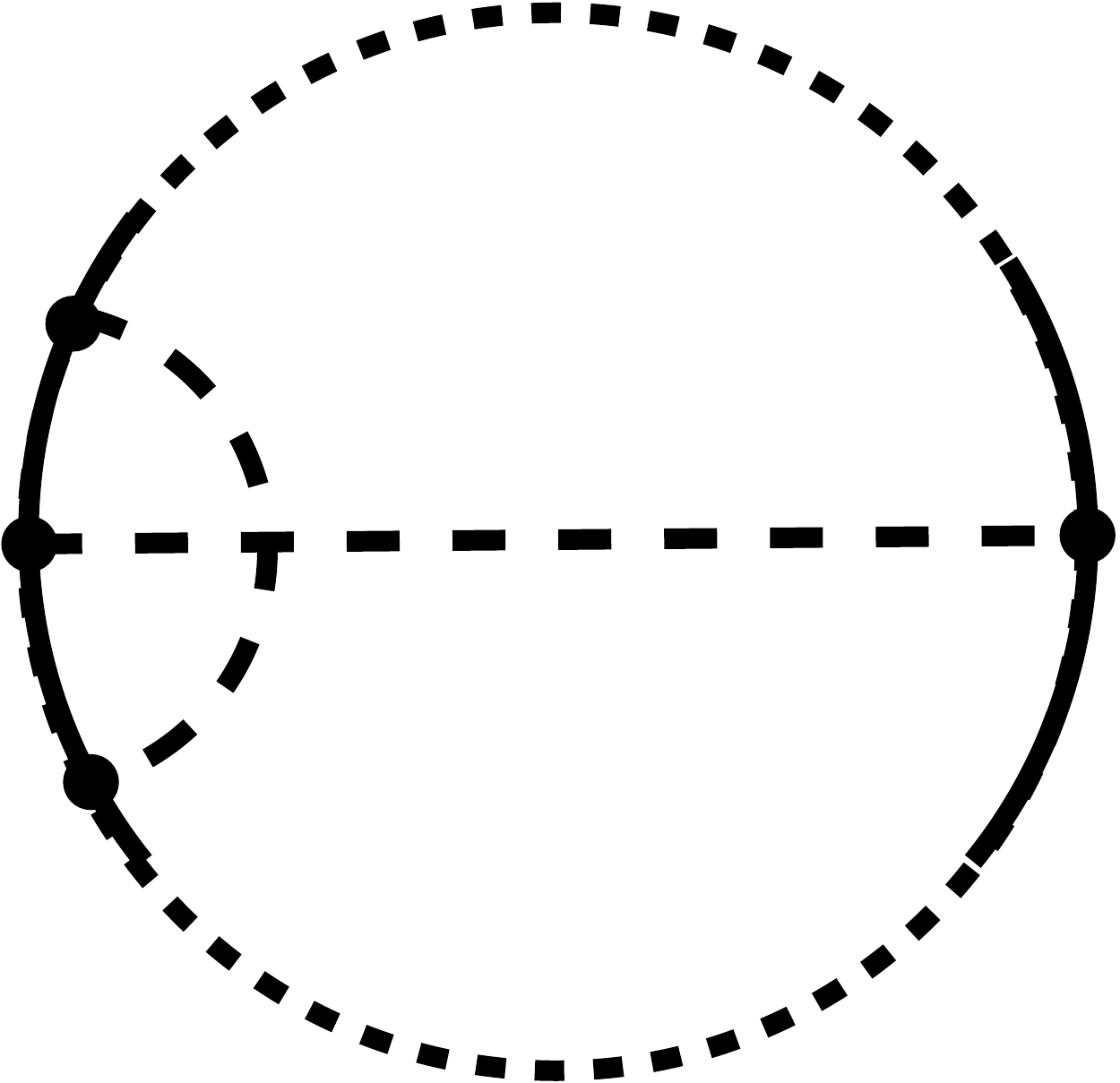} \raisebox{0.4cm}{-- 2} \includegraphics[width=1cm]{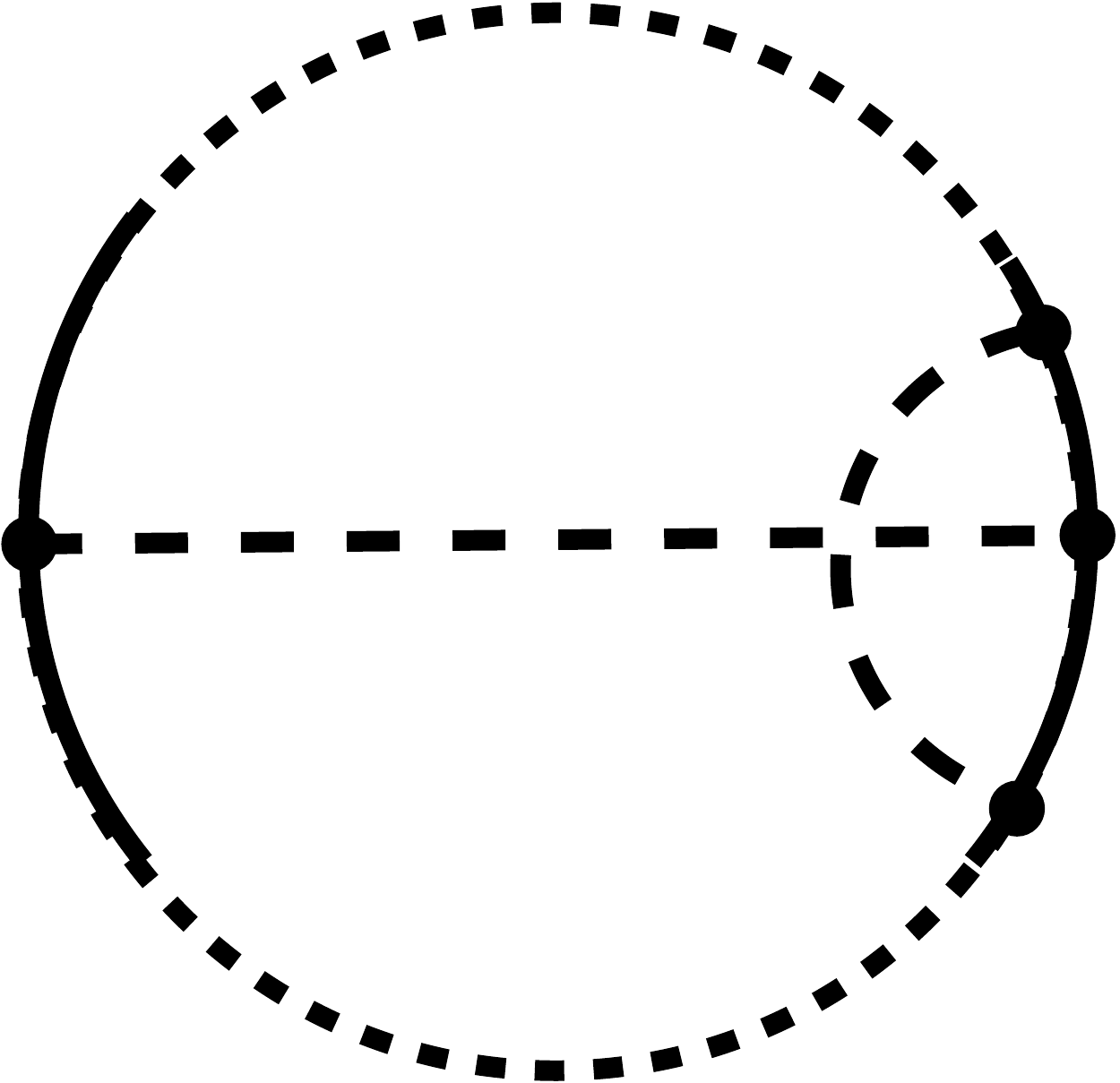} \raisebox{0.4cm}{--} \includegraphics[width=1cm]{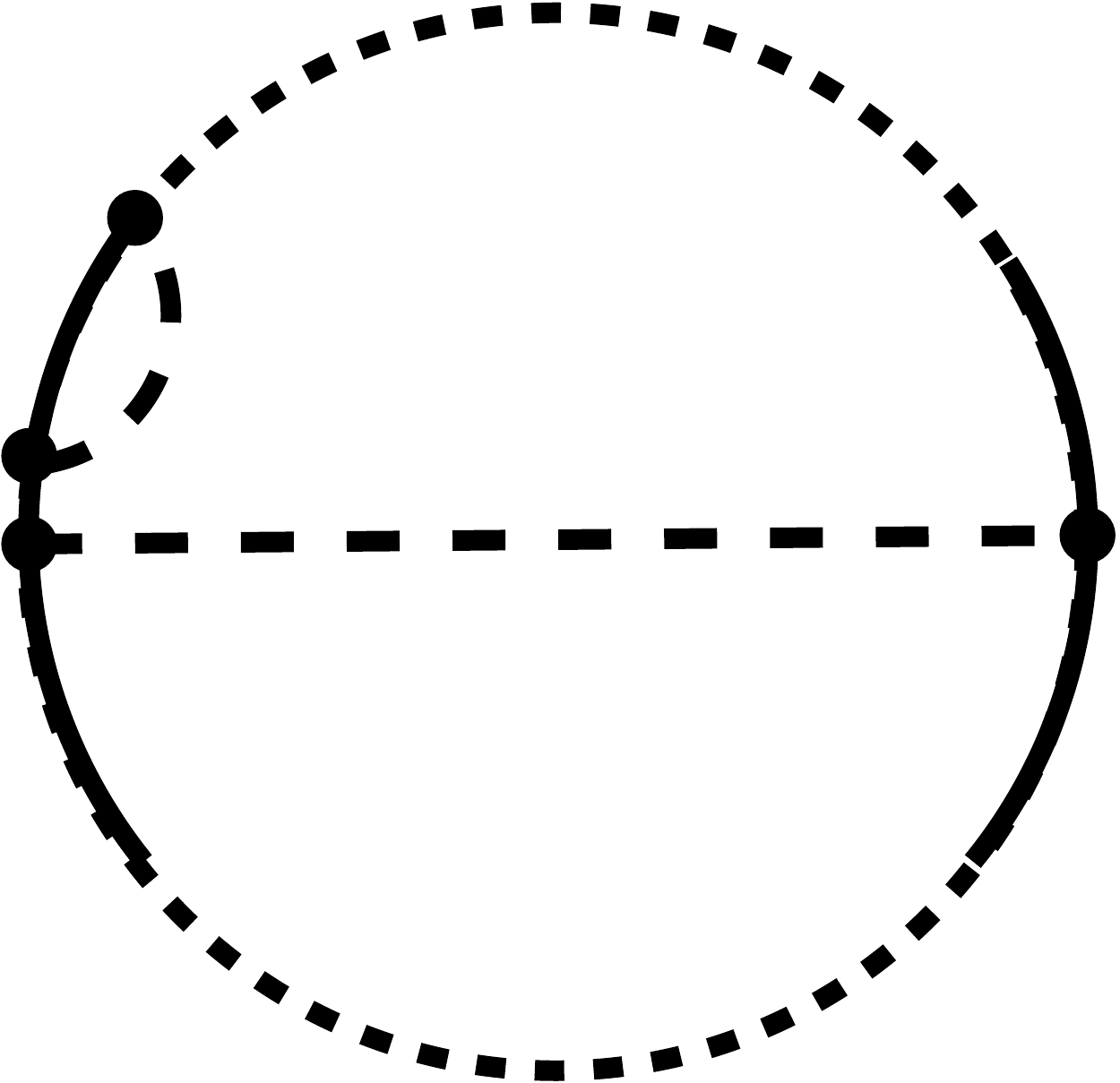} \raisebox{0.4cm}{--} \includegraphics[width=1cm]{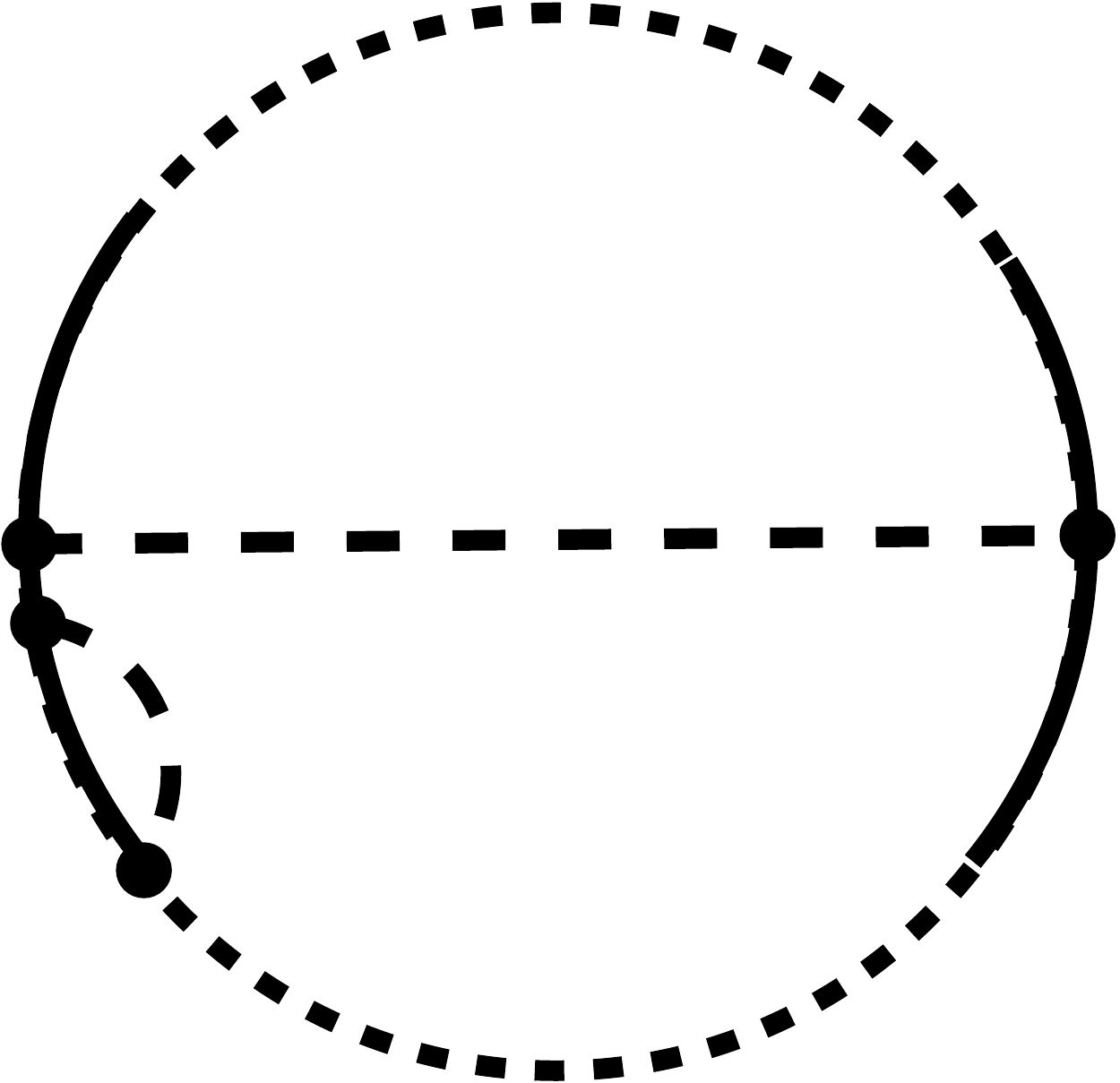} \raisebox{0.4cm}{+} \includegraphics[width=1cm]{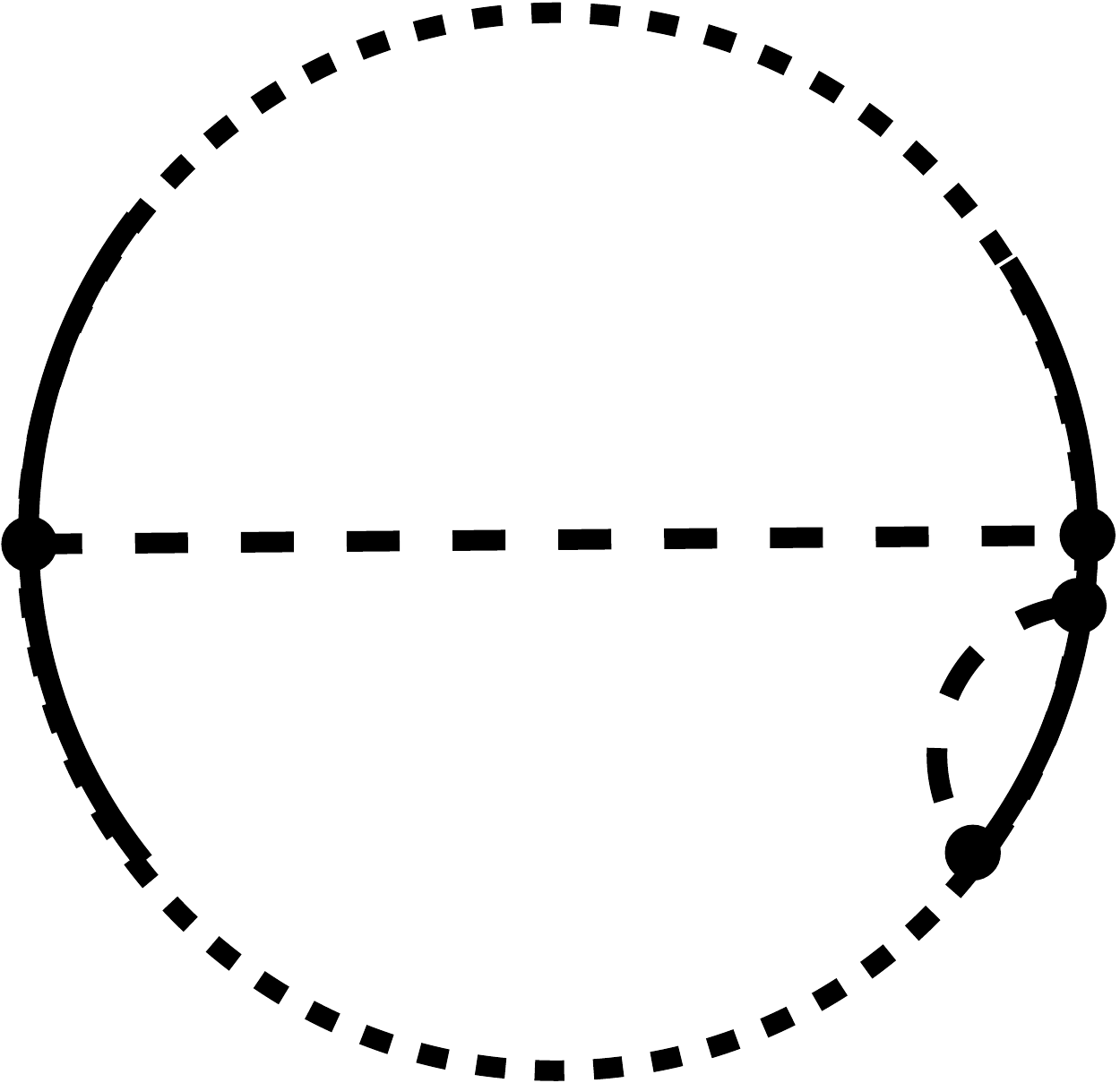}
\raisebox{0.4cm}{+} \includegraphics[width=1cm]{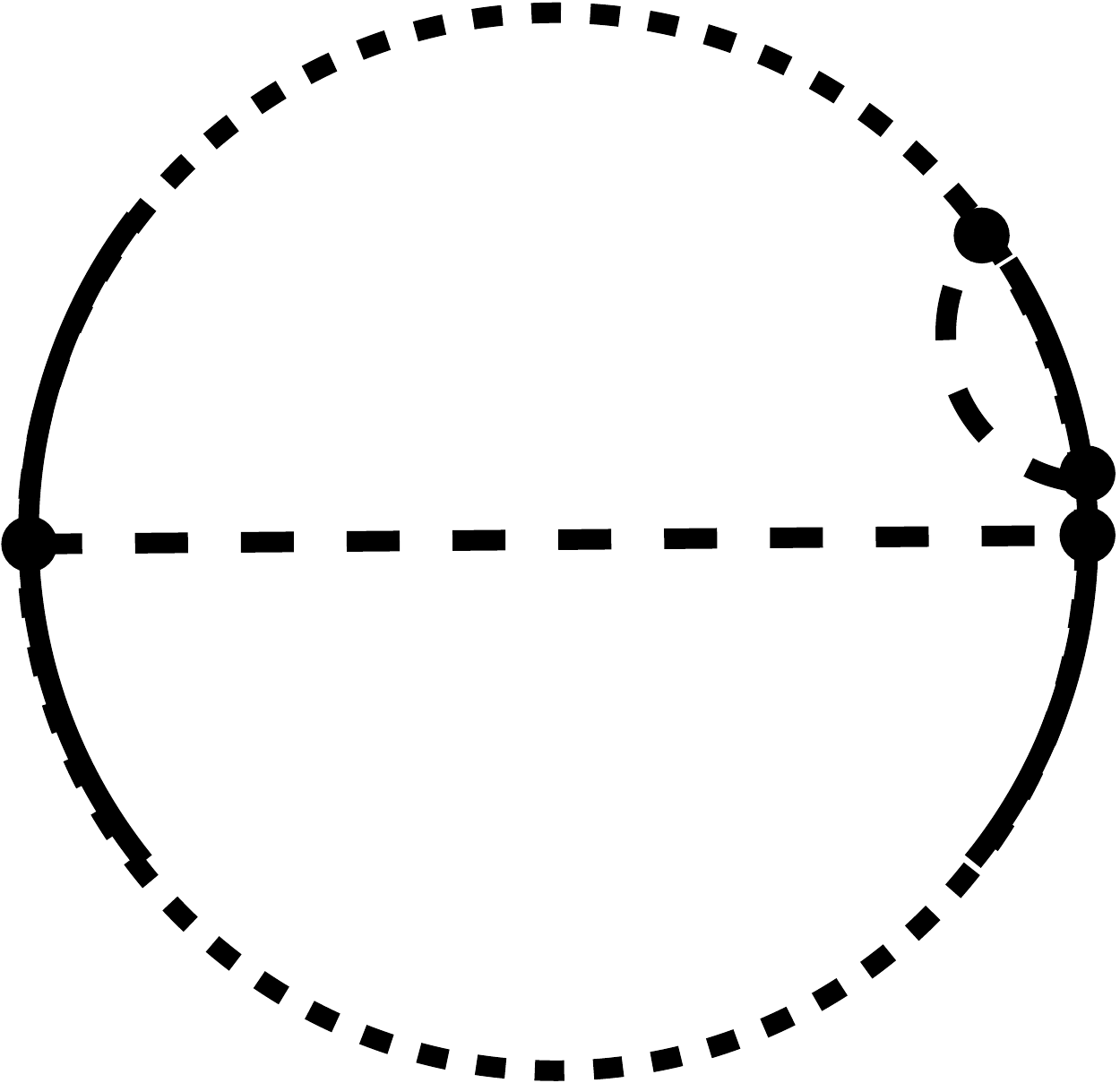}
 \raisebox{0.4cm}{=0}
\caption{Relations in $Co_\mathcal M$.}\label{white}
\end{center}
\end{figure}

The proof that the listed identities are satisfied is the same for each of the series: it is possible to split the diagrams taking part in a relation into 4 pairs. The difference of the diagrams forming a pair equals to the difference of a pair of diagrams with one oriented chord. The resulting linear combination of diagrams with one oriented chords vanishes identically.

The dimensions of the graded components of the spaces $\mathcal M$ and $Co_\mathcal M$ for small number of chords obtained by computer calculations made in Sage (\cite{Sage}) are
 \begin{center}
\begin{tabular}{c|c|c|c|c|c}
$n$ & 1 & 2 & 3 & 4 & 5\\
\hline
$\dim \mathcal M^n$ & 2 & 5 & 12 & 30 & 73  \\
\hline
$\dim Co^n_\mathcal M$ & 1& 2 & 5 & 12 & 29
\end{tabular}
\end{center}

The computational complexity of the problem grows very quickly with~$n$. For example, to calculate $\dim \mathcal M^5$ one needs to find the dimension of a subspace spanned by 20017 vectors in a vector space of dimension 3112.

A basis in the set of relations in $Co^4_\mathcal M$ is shown in Figure~\ref{basis}.

\begin{figure}[h]
\begin{center}

\includegraphics[width=1cm]{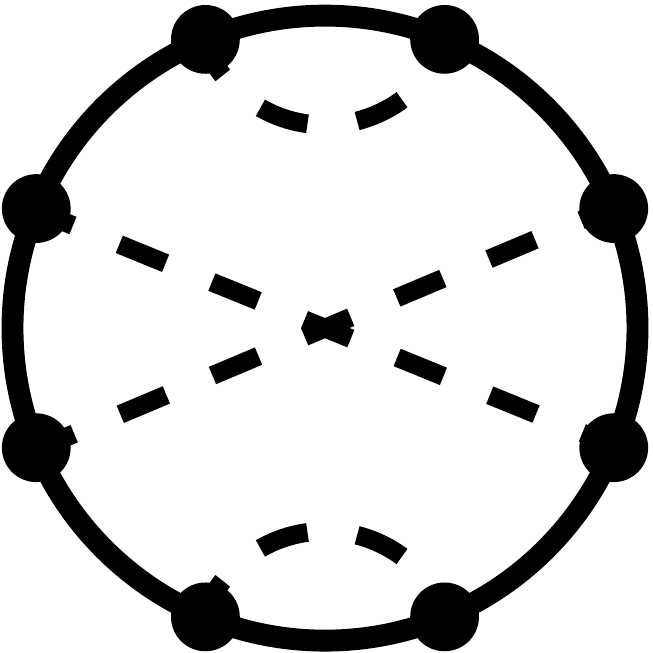} \raisebox{0.4cm}{--} \includegraphics[width=1cm]{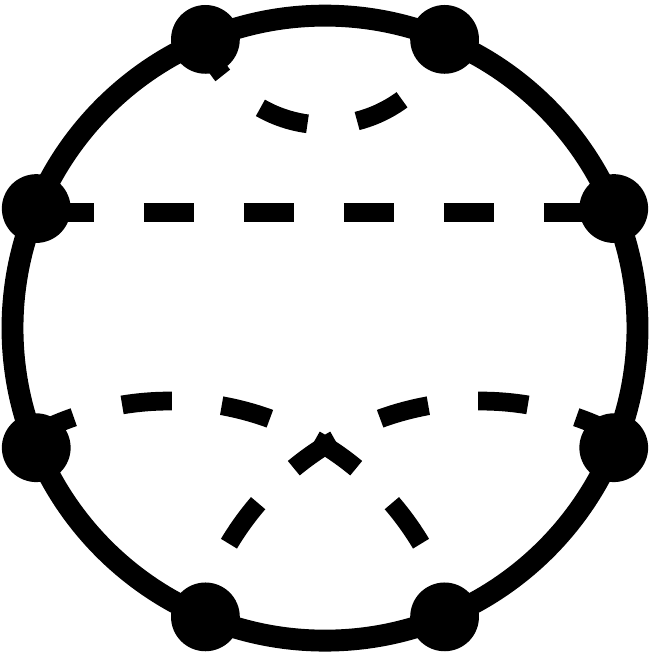}
\vspace{0.4cm}

\includegraphics[width=1cm]{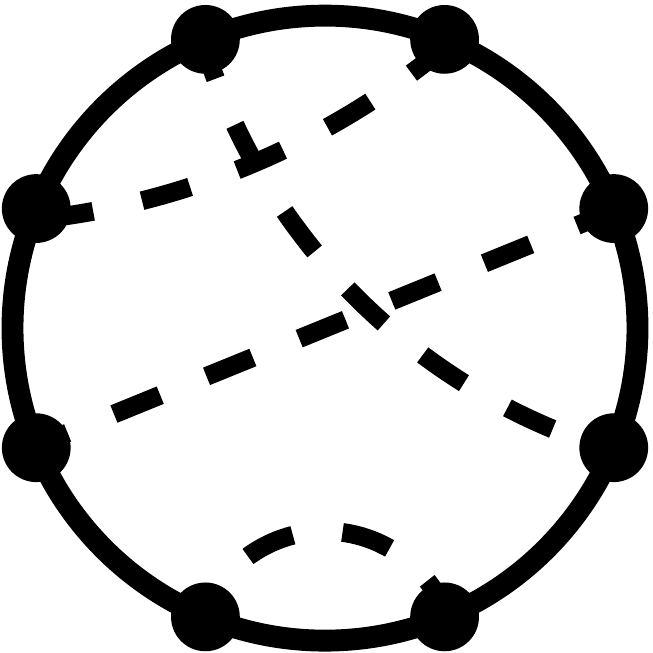} \raisebox{0.4cm}{--} \includegraphics[width=1cm]{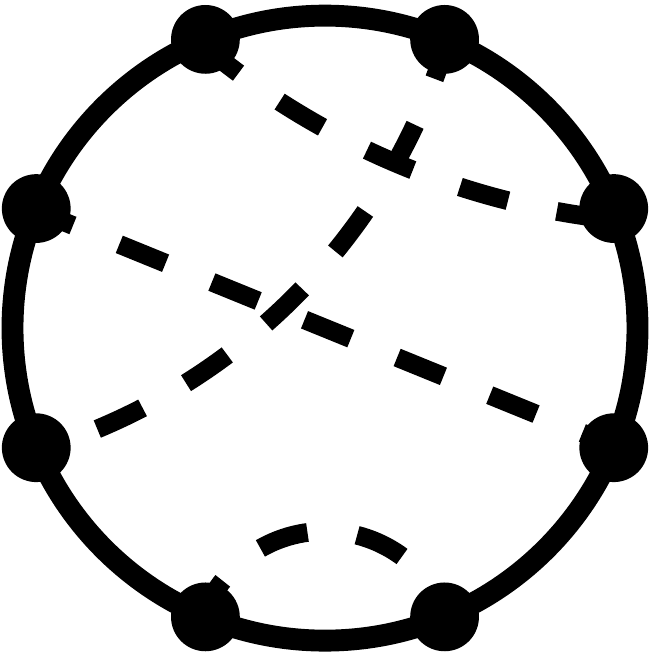}
\vspace{0.4cm}

\includegraphics[width=1cm]{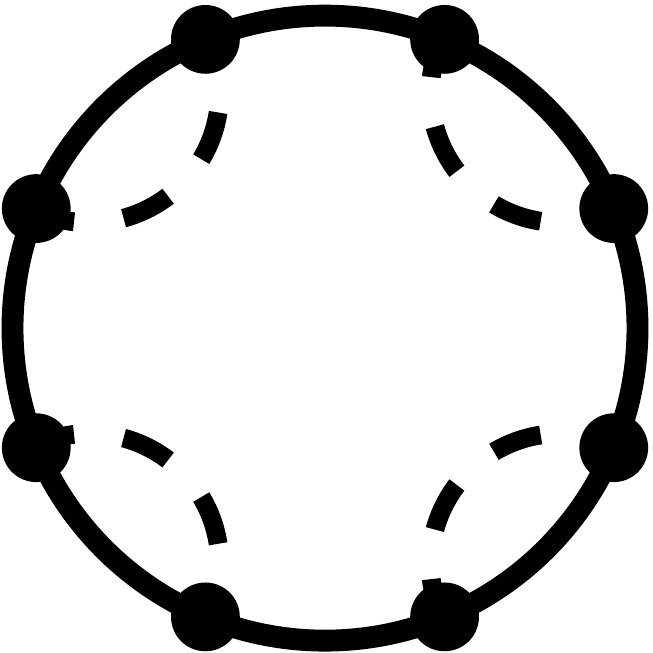} \raisebox{0.4cm}{-- 2} \includegraphics[width=1cm]{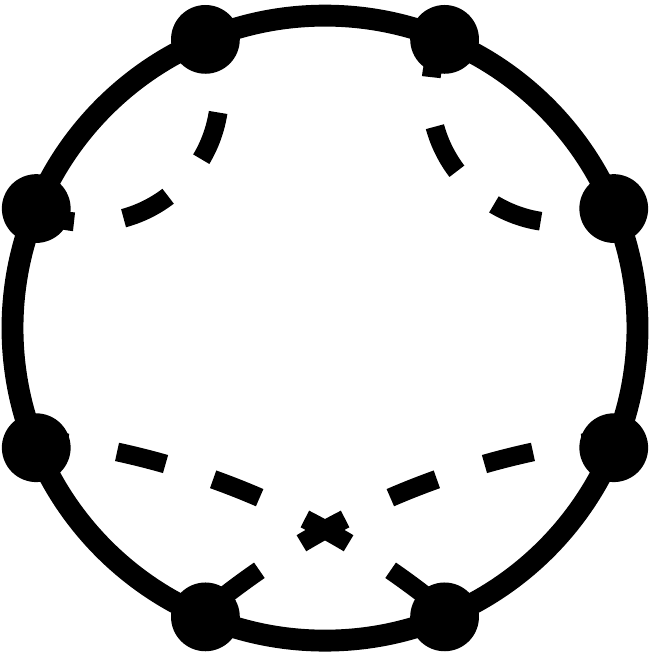} \raisebox{0.4cm}{--} \includegraphics[width=1cm]{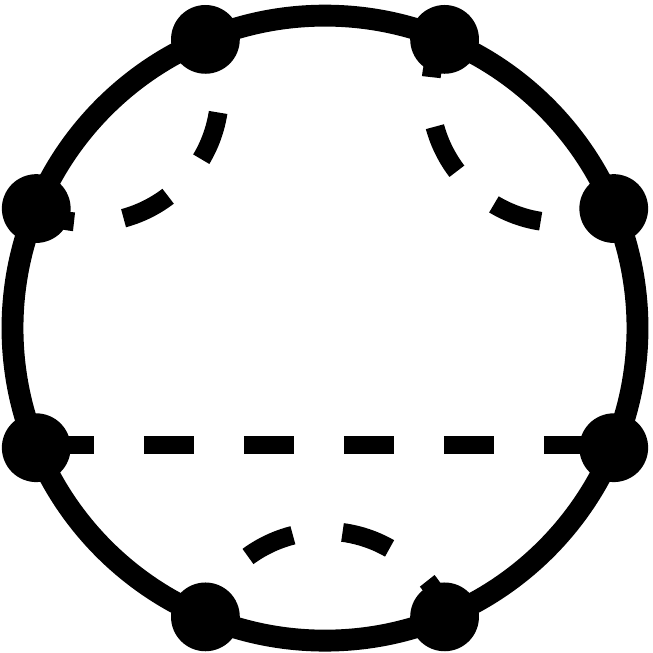} \raisebox{0.4cm}{+2} \includegraphics[width=1cm]{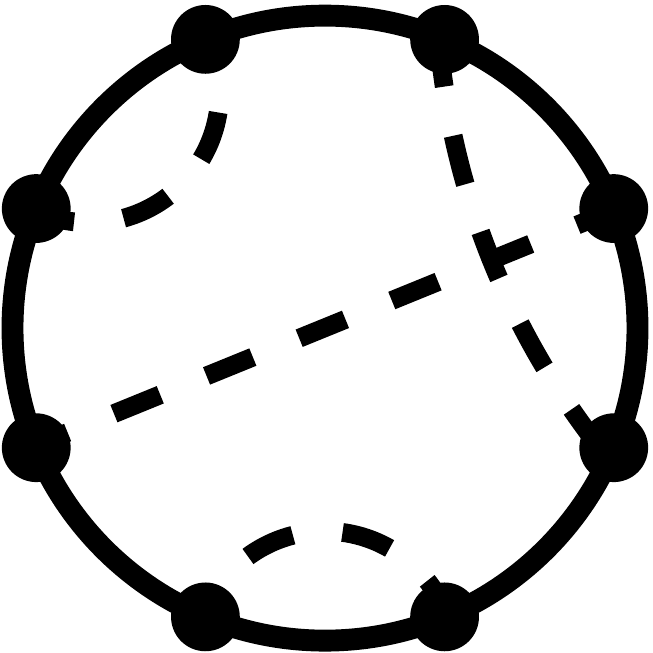} 
\vspace{0.4cm}

 \includegraphics[width=1cm]{cd4-02.pdf} \raisebox{0.4cm}{-- 2} \includegraphics[width=1cm]{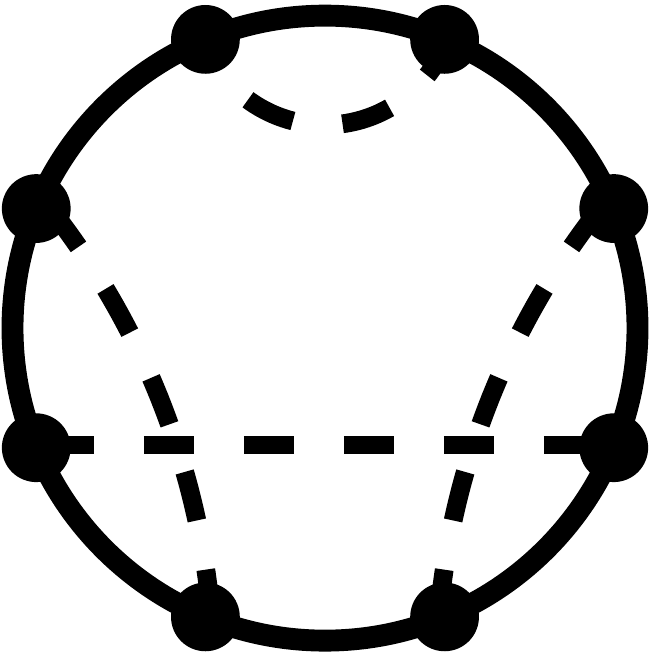} \raisebox{0.4cm}{+ 2} \includegraphics[width=1cm]{cd4-08.pdf} \raisebox{0.4cm}{--} \includegraphics[width=1cm]{cd4-10.pdf} 
\vspace{0.4cm}

 \includegraphics[width=1cm]{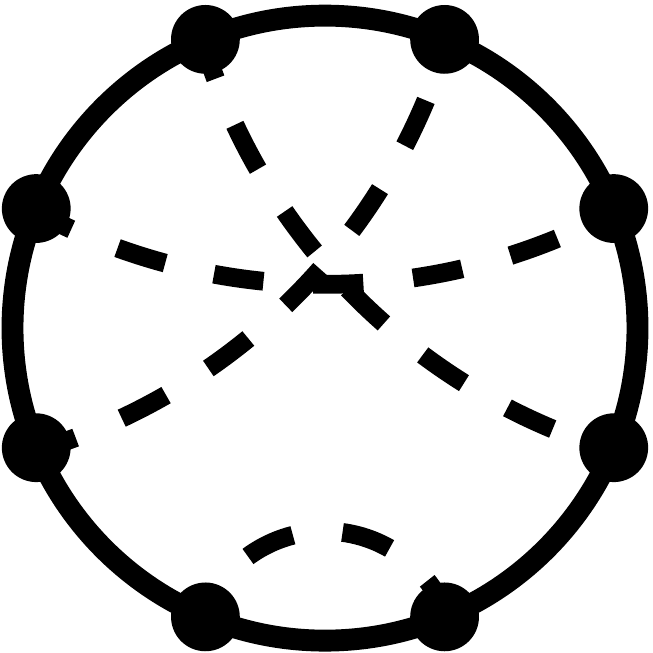} \raisebox{0.4cm}{-- 2} \includegraphics[width=1cm]{cd4-08.pdf} \raisebox{0.4cm}{+ } \includegraphics[width=1cm]{cd4-10.pdf} \raisebox{0.4cm}{--} \includegraphics[width=1cm]{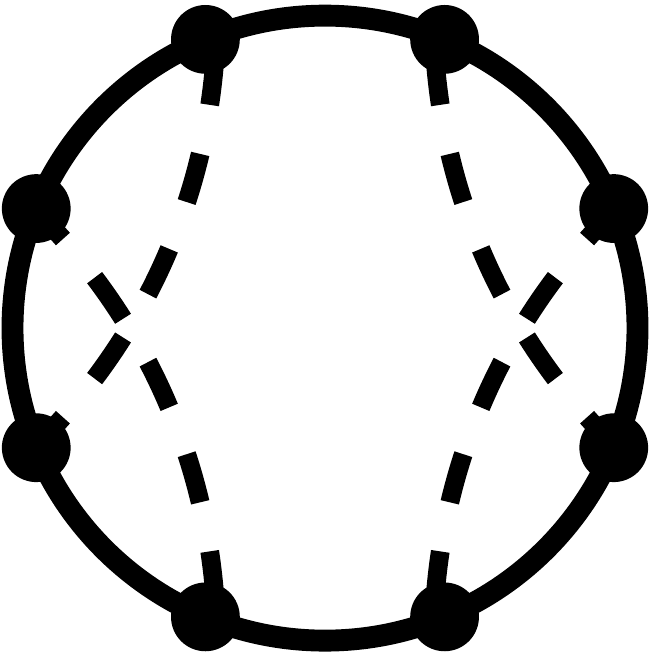} \raisebox{0.4cm}{+2}
 \includegraphics[width=1cm]{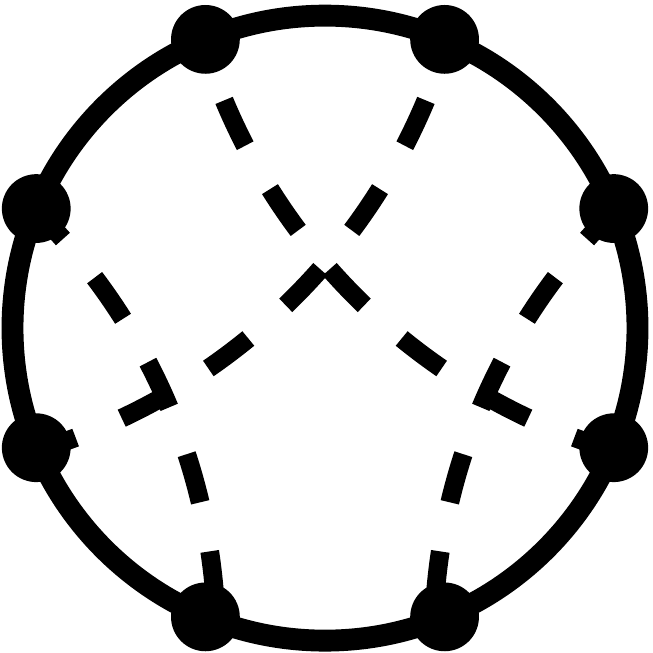} \raisebox{0.4cm}{--}
 \includegraphics[width=1cm]{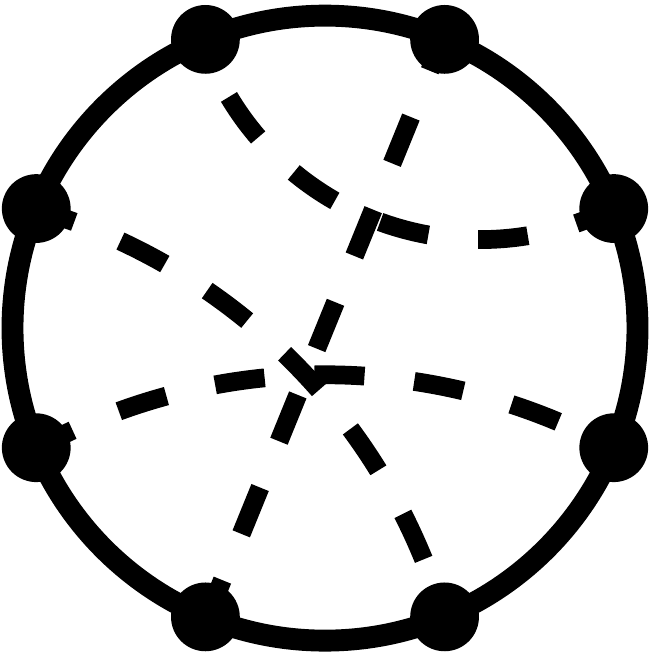}
\vspace{0.4cm}

\includegraphics[width=1cm]{cd4-08.pdf} \raisebox{0.4cm}{-- 2} \includegraphics[width=1cm]{cd4-13.pdf} \raisebox{0.4cm}{+ 2} \includegraphics[width=1cm]{cd4-14.pdf} \raisebox{0.4cm}{--} \includegraphics[width=1cm]{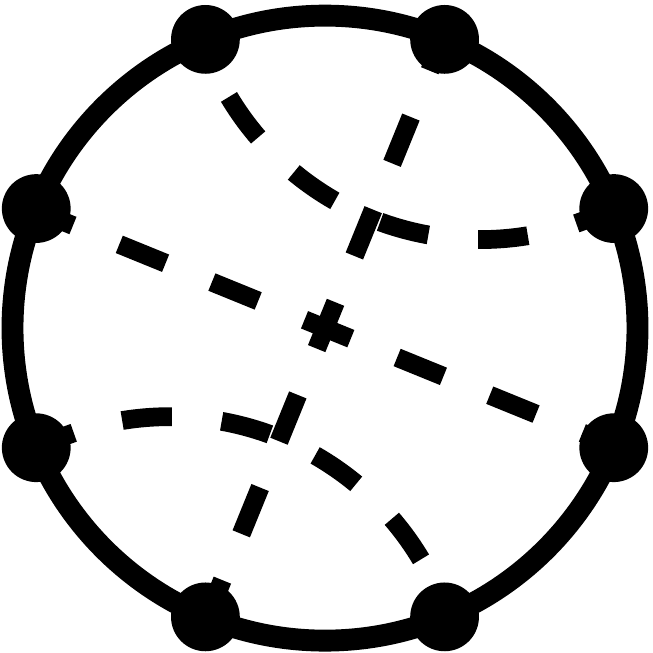} \raisebox{0.4cm}{--}
 \includegraphics[width=1cm]{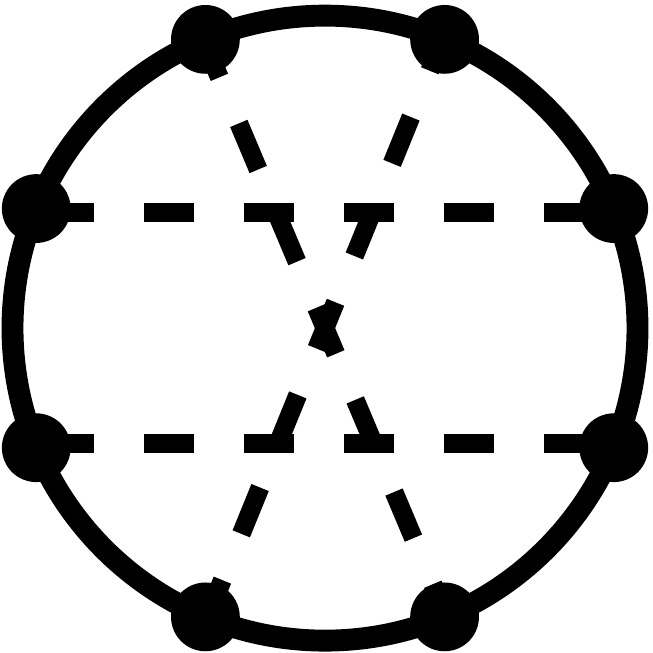} \raisebox{0.4cm}{+}
 \includegraphics[width=1cm]{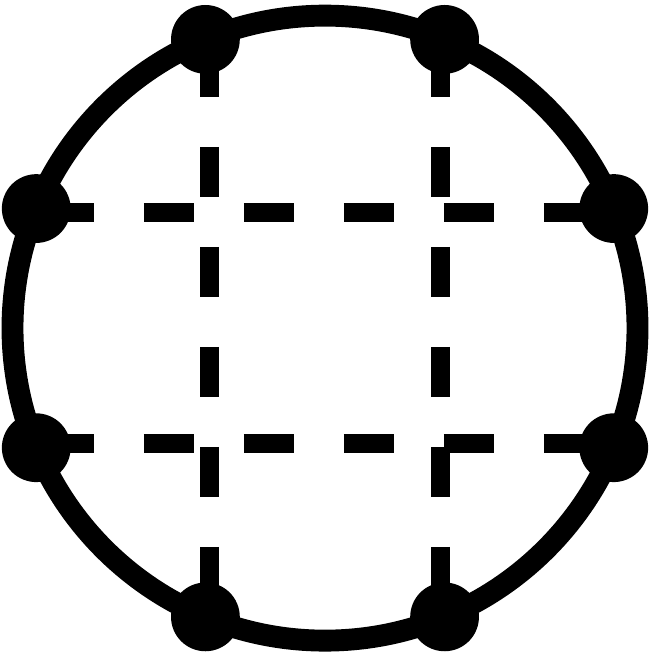}
 
\caption{A basis in the space of relations in  $Co^4_\mathcal M$.}\label{basis}
\end{center}
\end{figure}

\section{Framed graphs}

Recall that the \emph{Lando graph algebra} (see~\cite{LandoHopf}) over $\mathbb K$ is defined as the vector space spanned by all possible finite simple graphs modulo the 4T-relations usually written in the form
$$\Gamma -\Gamma'_{uv} = \tilde \Gamma_{uv} - \tilde \Gamma'_{uv}.$$

This formula has the following meaning. Let $\Gamma$ be a graph, and let $u, v$ be two its vertices joined by an edge. Then $\Gamma'_{uv}$ denotes the graph obtained from $\Gamma$ by erasing the edge between $u$ and $v$. The graph $\tilde \Gamma_{uv}$ is obtained from $\Gamma$ by the following operation: for every vertex $w$ different from $u,v$ and connected by an edge with $v$, the vertices $u$ and $w$ are joined by an edge in $\tilde \Gamma_{uv}$ if and only if the vertices $u$ and $w$ are not joined in $\Gamma$. The adjacencies of all the other possible pairs of vertices in $\Gamma$ and $\tilde \Gamma_{uv}$ are the same. Finally, $\tilde \Gamma'_{uv}$ is obtained from $\tilde \Gamma_{uv}$ by erasing the edge between $u$ and $v$. An example of a 4T-relation is shown in Figure~\ref{graphs}.
\begin{figure}[ht]
\begin{center}
\includegraphics[width=8cm]{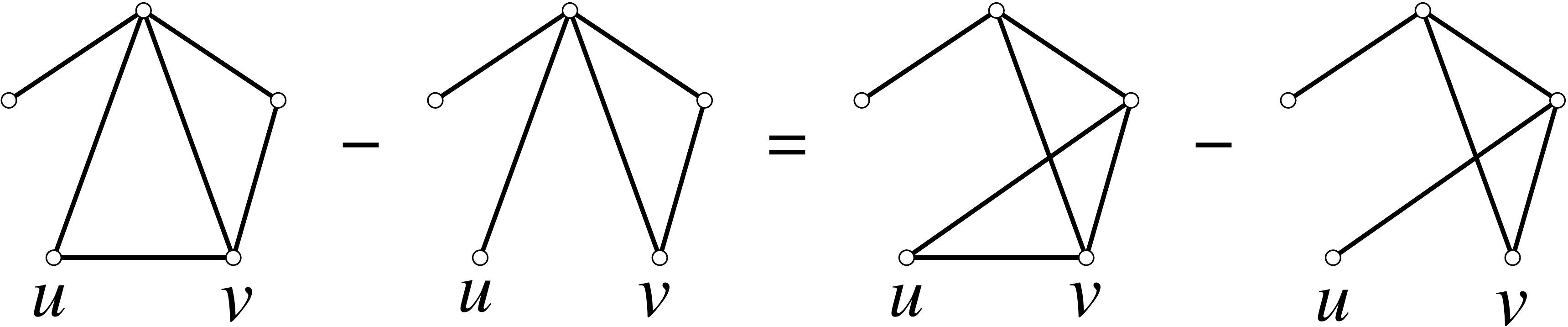}
\end{center}
\caption{An example of a 4T-relation.}\label{graphs}
\end{figure}

The vector space of all finite graphs modulo the 4T-relations is denoted by $\mathcal G$. The linear maps
$$m_\mathcal G \colon \mathcal G\otimes \mathcal G \to \mathcal G,$$
taking the tensor product of two graphs to their disjoint union, and
$$\Delta_\mathcal G \colon \mathcal G \to \mathcal G\otimes \mathcal G,$$
taking a graph $\Gamma$ to
$$\Delta_\mathcal G (\Gamma) = \sum \Gamma_I\otimes \Gamma_J,$$
where the sum is taken over all possible ways to decompose the set of vertices of $\Gamma$ into the union of two non-intersecting sets $I$ and $J$, and $\Gamma_I$ is a complete subgraph of $\Gamma$ on the set of vertices $I$, along with natural operations of unit, counit and antipode, endow $\mathcal G$ with the structure of graded commutative cocommutative graded Hopf algebra. The grading of a graph is the number of its vertices.

The framed version of the Lando graph algebra was also introduced in~\cite{Lando}.

\begin{defn} By a \emph{framed graph} $\Gamma$ we mean a finite simple graph along with a fixed \emph{framing}. A framing is a map from the set of vertices of the graph to $\mathbb Z/2\mathbb Z$.
\end{defn}

In the following, we are going to refer to the vertices of framing $0$ as \emph{black vertices}, and to the vertices of framing $1$ as \emph{white vertices}.

The $\mathbb K$-linear space spanned by framed graphs modulo the \emph{framed 4T-relations} (see below) is denoted by $\mathcal H$. In notation introduced earlier, the framed 4T-relations read
$$\Gamma -\Gamma'_{uv} = (-1)^{f(v)}(\tilde \Gamma_{uv} - \tilde \Gamma'_{uv}),$$
where $f(v)$ denotes the framing of the vertex $v$ in $\Gamma$, and the framing of the vertex $u$ of graphs $\tilde \Gamma_{uv}$ and $\tilde \Gamma'_{uv}$ equals $f(u)+f(v)$.

Similarly to $\mathcal G$, the space $\mathcal H$ admits operations of a product $m_\mathcal H$ (induced by the disjoint union of graphs), and a coproduct $\Delta_\mathcal H$ taking a framed graph $\Gamma$ to
$$\Delta_\mathcal H (\Gamma) = \sum \Gamma_I\otimes \Gamma_J,$$ where summation again is carried 
over all possible choices of a subset $I$ of vertices of $\Gamma$, and the framing of $\Gamma_I$ 
is supposed to be the restriction of the framing of $\Gamma$.

\begin{thm}[Lando \cite{Lando}]
The operations $m_\mathcal H$ and $\Delta_\mathcal H$, along with natural maps of unit, counit, and antipode (whose definitions are similar to those for $\mathcal G$), endow $\mathcal H$ with the structure of graded commutative cocommutative Hopf algebra.
\end{thm}

\begin{defn}
Let $\Gamma$ be a framed graph. We say that $\Gamma$ \emph{is black} if all its vertices are black, and we say that it \emph{is white} if all its vertices are white.
\end{defn}

Note that there is a Hopf algebras monomorphism $\mathcal G\to \mathcal H$.

Define the map $Pr_\mathcal H$ as
$$Pr_\mathcal H (\Gamma) = \left\{\begin{array}{l} \Gamma,\mbox{ if $\Gamma$ is a black graph,}\\					
						0, \mbox{ otherwise.}\end{array}\right.$$
 Put
$$\delta_\mathcal H = (Pr_\mathcal H \otimes id)\circ \Delta_\mathcal H.$$

As in the case of framed chord diagrams, the following theorem holds.

\begin{thm}
The space $\mathcal H$, along with the maps $m_\mathcal H|_{\mathcal G\otimes \mathcal H}$ and $\delta_H$, is a $\mathcal G$-Hopf module.
\end{thm}
\begin{proof}
First, we have to check that $\mathcal H$ is a $\mathcal G$-comodule. The equality
$$ (\Delta_\mathcal G\otimes id) \circ \delta_\mathcal H = (id\otimes \delta_\mathcal H)\circ \delta_\mathcal H $$
is obvious, since both sides of the equality correspond to the decompositions of the set of black vertices of a diagram into three disjoint subsets. The equality
$$\iota\circ (\epsilon_\mathcal G\otimes id)\circ\delta_\mathcal H=id$$
is also straightforward, due to the definition of $\epsilon_\mathcal G$.

Similarly to the case of framed chord diagrams, the value of $\delta_\mathcal H$ on a product of a black graph and an arbitrary framed graph is expressible as a sum over all possible ways to choose a collection of black vertices from both factors. The equality $\delta_\mathcal H \circ m_\mathcal H (\Gamma_1\otimes \Gamma_2) = m_{\mathcal G\otimes \mathcal H} (\Gamma_1\otimes \delta_\mathcal H (\Gamma_2))$ for any black graph $\Gamma_1$ and any framed graph $\Gamma_2$ is obvious.
\end{proof}

The Larson-Sweedler theorem implies that $\mathcal H$ as a $\mathcal G$-Hopf module is isomorphic to $\mathcal G\otimes Co_\mathcal H$, where $Co_\mathcal H$ is the covariant submodule of $\mathcal H$.

\begin{thm}
The space $Co_\mathcal H$ is spanned by white graphs.
\end{thm}
\begin{proof}
The strategy of the proof is exactly the same as in the case of framed chord diagrams. It is clear that the subspace spanned by white graphs is contained in $Co_\mathcal H$.

Now we show that any framed graph $\Gamma$ is representable as a linear combination of products of a black graph by a white graph. Take an arbitrary framed graph. If all its vertices have the same framing, then we are done.

Suppose now that not all the vertices of $\Gamma$ have the same framing. Define the \emph{complexity} of the graph as
$$C_\Gamma = 2^b\frac{(b-1)!}{w!}N_\Gamma,$$
where $b$ is the number of black vertices, $w$ is the number of white vertices, and $N_\Gamma$ is the number of edges of $\Gamma$ that connect vertices of different framing.

As in the case of framed chord diagrams, if the complexity is 0, then the graph is clearly representable as a product of a black graph by a white graph. The minimal possible non-zero complexity of $\Gamma$ is $2^b\frac{(b-1)!}{w!}$

Let the complexity of $\Gamma$ is greater than zero. Take an edge connecting two vertices of different framings, and apply the 4T-relation to this edge, using the white vertex as the vertex $v$. We obtain the presentation of $\Gamma$ as a linear combination of three graphs. One of these graphs is obtained from $\Gamma$ by deletion of the edge connecting vertices of different framings, 
thus its complexity is smaller than that of~$G$. Two other graphs have one white vertex more than $\Gamma$. The maximal possible complexity of a graph with $(b-1)$ black vertices and $(w+1)$ white vertices is
$$ 2^{b-1}\cdot (b-1)(w+1)\cdot \frac{(b-2)!}{(w+1)!} = 2^{b-1}\frac{(b-1)!}{w!}$$
which is definitely less than  $2^b\frac{(b-1)!}{w!}$.
\end{proof}

Similarly to the case of $\mathcal M$, the space $\mathcal H$ admits a bigrading coming from its representation as the tensor product $\mathcal G\otimes Co_\mathcal H$.

Note that the space $Co_\mathcal H$ is not just a subspace of $\mathcal H$, it is a Hopf subalgebra with the product, coproduct, unit, counit and antipode given by the restrictions of the corresponding operations to $\mathcal H$. Therefore, due to the Milnor-Moore theorem, it is a polynomial algebra with generators given by primitive elements.

Similarly to the case of white diagrams,  the set of white graphs is not a basis of $\mathcal H$. The dimensions of graded components of the spaces $\mathcal G$ and $\mathcal H$ (calculated by I.~Dynnikov and enlisted in~\cite{LandoShkol}), and the numbers of simple unlabeled graphs on $n$ vertices (Sloane's A000088, see~\cite{Sloane}), prescribe the existence of a non-trivial relation involving white graphs on 4 vertices.

Unfortunately, we do not know a general form of a non-trivial relation in the space $Co_\mathcal H$

Using the dimensions of graded components of $\mathcal G$ and $\mathcal H$, we were able to calculate the dimensions of graded components of $Co_\mathcal H$ and its primitive subspace $PCo_\mathcal H$ for small $n$ (for the reader's convenience,
the dimensions of graded components of $\mathcal H$ from~\cite{LandoShkol} are also reproduced):
\begin{center}
\begin{tabular}{c|c|c|c|c|c}
$n$ & 1 & 2 & 3 & 4 & 5\\
\hline
$\dim \mathcal H^n$ & 2 & 5 & 11 & 26 & 58 \\
\hline
$\dim Co^n_\mathcal H$ & 1 & 2 & 4 & 9 & 19  \\
\hline
$\dim PCo^n_\mathcal H$ & 1& 1 & 2 & 4 & 8
\end{tabular}
\end{center}

\section{Framed weight systems and the intersection graph of a framed chord diagram and}

Recall that a linear map from $\mathcal A$ to an abelian group $A$ is traditionally called an \emph{$A$-valued weight system}. Weight systems are important objects; they allow one to construct finite-type knot invariants, due to the Vassiliev--Kontsevich theorem (see, for example~\cite{CDbook}).

\begin{defn}
Let $A$ be an abelian group. A linear map from the space $\mathcal M$ to $A$ is called an \emph{$A$-valued framed weight system}.
\end{defn}

Gven an arbitrary weight system, it is possible to construct a framed weight system. Recall, that for the description of the $\mathcal A$-Hopf module structure on $\mathcal M$, we exploited the projection operator $Pr_{\mathcal M}$. Let us introduce one more projection operator $\mathcal M \to \mathcal A$. Define the action of the \emph{discoloration operator} $D_\mathcal M \colon \mathcal M \to \mathcal A$ on a framed chord diagram $d$ by the rule:
$$ D_\mathcal M (d) = (-1)^w d',$$
where $w$ is the number of disorienting chords of $d$, and $d'$ is obtained from $d$ by setting the framings of all the chords to be 0. Extend $D_\mathcal M$ to $\mathcal M$ by linearity.

\begin{prop}
The operator $D_\mathcal M$ is well-defined.
\end{prop}

\begin{proof}
Straightforward, as the image of an arbitrary framed 4T-relation under $D_\mathcal M$ is a non-framed 4T-relation.
\end{proof}

\emph{Remark:} The discoloration operator can also be used to define a comultiplication operator $\delta'_\mathcal M\colon \mathcal M \to \mathcal A\otimes \mathcal M$ as
$$\delta'_\mathcal M = (D_\mathcal M \otimes id)\circ \Delta_\mathcal M.$$
The operator $\delta'_\mathcal M$, along with $m_\mathcal M$, provides $\mathcal M$ with one more $\mathcal A$-Hopf module structure, different from the one we have actually used.
\vspace{0.3cm}    

It is easy to check, that given an arbitrary non-framed weight system $w$, the composition $w\circ D_\mathcal M$ is a framed weight system.

Another example of a framed weight system can be given using the notion of the \emph{intersection graph}.

The intersection graph of a (non-framed) chord diagram $d$ is defined as the graph $\Gamma_d$ whose vertices are in one-to-one correspondence with the chords of $d$, and two vertices are connected by an edge if the corresponding chords intersect (see~\cite{CDbook}). This construction was first proposed by S. Chmutov, S. Duzhin, and S. Lando  in 1994.

\begin{defn}
Let $d$ be a framed chord diagram. The \emph{intersection graph of $d$} is the graph $\Gamma_d$ whose vertices are in one-to-one correspondence with the chords of $d$. Two vertices of $\Gamma$ are connected by an edge if and only if the corresponding chords of $d$ intersect (this means that their endpoints $a_1, a_2$ and $b_1,b_2$ appear in the interlacing order $a_1,b_1,a_2,b_2$ while travelling along the circle). The framing of $\Gamma_d$ is inherited from that of $d$.
\end{defn}

Define the map $I\colon \mathcal M \to \mathcal H$ as the linear extension of the correspondence $d\mapsto \Gamma_d$. We use the same letter for the restriction of this map $I\colon \mathcal A\to \mathcal G$.

\begin{thm}
The map $I$ is a well-defined linear map. The squares
$$
\xymatrix{
\mathcal A \otimes \mathcal M \ar[r]^{I\otimes I} \ar[d]^{m_\mathcal M} & \mathcal G\otimes \mathcal H \ar[d]^{m_\mathcal H}\\
\mathcal M \ar[r]^I & \mathcal H
}$$
and
$$
\xymatrix{
\mathcal M \ar[r]^{I} \ar[d]^{\delta_\mathcal M} & \mathcal H \ar[d]^{\delta_\mathcal H}\\
\mathcal A\otimes \mathcal M \ar[r]^{I\otimes I} & \mathcal G\otimes \mathcal H
}\hspace{1cm}
\xymatrix{
\mathcal M \ar[r]^{I} \ar[d]^{\Delta_\mathcal M} & \mathcal H \ar[d]^{\Delta_\mathcal H}\\
\mathcal M\otimes \mathcal M \ar[r]^{I\otimes I} & \mathcal H\otimes \mathcal H
}$$
are commutative.
\end{thm}
\begin{proof}
We have to check that a linear combination of framed chord diagrams vanishing due to the 4T-relation maps to a linear combination of framed graphs vanishing due to the 4T-relations as well. It is a straightforward generalization of the corresponding statement in the non-framed case, and the proof is the same.
It is clear that the multiplication of chord diagrams corresponds to the disjoint union of graphs. This implies the commutativity of the first square. The commutativity of the second and the third square is also straightforward.
\end{proof}

Using $I$, we can construct framed weight systems. Indeed, let $w$ be a linear map from $\mathcal H$ to an abelian group $A$. Then $w\circ I$ is obviously a framed weight system.

Let us recall that the \emph{chromatic polynomial} $ch$ is a well-defined map from $\mathcal G$ to $\mathbb K[x]$ (this fact was first remarked by S.V.~Duzhin).

A natural extension of the chromatic polynomial to the framed graphs takes values in the polynomial algebra $\mathbb K[x,y]$.  Let $\Gamma$ is a framed graph, $e$ is an edge of $\Gamma$, and $y$ and $z$ are the vertices adjacent to~$e$. We define the framed chromatic polynomial, which we also denote by $ch$, via the skein-relation
$$ ch(\Gamma) = ch(\Gamma'_e) - (-1)^{f(y)+f(z)} ch(\Gamma''_e),$$
and the additional relations
$$ ch(\Gamma \sqcup *_b) = x\cdot ch(\Gamma),$$
$$ ch(\Gamma \sqcup *_w) = y\cdot ch(\Gamma),$$
$$ ch (\emptyset) = 1$$
where the graph $\Gamma'_e$ is obtained from $\Gamma$ by deleting the edge $e$, and the graph $\Gamma''_e$ is obtained from $\Gamma$ by contracting $e$. The framings of the vertices of $\Gamma'_e$, and all the vertices of $\Gamma''_e$, except for the vertex $v$ appearing as the result of edge contraction, are inherited from that of the corresponding vertices of $\Gamma$. In the case of contraction, we set $f(v)=0$ if $f(x)=f(y)=0$, and 1 otherwise\footnote{Formally, $f(v) = f(y)+f(z)+f(y)f(z)$.}. The symbols $*_b,*_w$ and $\emptyset$ denote the graph with one black vertex and no edges, the graph with one white vertex and no edges, and the empty graph, respectively. The value of the chromatic polynomial on any graph containing loops is 0.

Note, that if we restrict our attention to the case of non-framed graphs, then we obtain the usual definition of the chromatic polynomial.

Extend the correspondence $\Gamma \mapsto ch(\Gamma)$ by linearity.

\begin{prop}
The framed chromatic polynomial is defined by the above presented skein-relations uniquely. It extends to a well-defined linear map $ch\colon \mathcal H \to \mathbb K[x,y]$.
\end{prop}
\begin{proof}
Consider a framed graph $\Gamma$. First note that using the first skein relation one is able to present $\Gamma$ as a linear combination of edgeless graphs. We need to check that the resulting combination does not depend on the order we get rid of the edges when applying the skein relation. This check is straightforward.

Now it is clear that the framed chromatic polynomial is uniquely defined by the skein relation.

Finally, we have to check that the appropriate linear combination of the values of the framed chromatic polynomial on a 4-tuple of graphs taking part in a 4T-relation, vanishes. As in the non-framed case this follows from the fact that both sides of a 4T relation are the same modulo the skein relation defining $ch$.
\end{proof}

\emph{Remark:} The polynomial algebra $\mathbb K[x,y]$ also admits a Hopf algebra structure (with $x$ and $y$ being the primitive generators of degree 1). The chromatic polynomial $ch\colon \mathcal H\to \mathbb K[x,y]$ is a morphism in the category of Hopf algebras. \vspace{0.3cm}

Recall, that given a $A$-valued weight system $w$, and $A'$-valued weight system $w'$, the composition $(w\otimes w')\circ \Delta_\mathcal A$ is an $A\otimes A'$-valued weight system (the \emph{convolution product} of weight systems). Using the comodule structure $\delta_\mathcal M$ on $\mathcal M$ we can extend this operation to produce framed weight system   

\begin{thm}
Let $A,A'$ are abelian groups,  $w$ is an $A$-valued weight system, and $w'$ is an $A'$-valued framed weight system. Then the composition map
$$ (w\otimes w') \circ\delta_\mathcal M\colon \mathcal M \to A\otimes A'$$
is an $A\otimes A'$-valued framed weight system.
\end{thm}
\begin{proof}
Trivial, since the map under consideration is a composition of well-defined linear maps.
\end{proof}

There are weight systems that are not determined by the intersection graph only. For example, there exists an order 11 weight system, that distinguishes two mutant knots~\cite{DistMutKnots}, and thus, by a theorem of Chmutov and Lando (see~\cite{ChL}), can not be obtained from a linear map from $\mathcal G$. Note, that even putting in the previous Theorem $w' = ch\circ I$, by a suitable choice of $w$, it is possible to construct framed weight system not determined by the intersection graph only.


\thebibliography{99}
\bibitem{Arnold} V.I. Arnold, \emph{Plane curves, their invariants, perestroikas and classifications.} In \emph{Singularities and bifurcations}, volume 21 of Adv. Soviet Math., 33-91. Amer. Math. Soc., Providence, RI, 1994. With an appendix by F. Aicardi.

\bibitem{CDbook} S. Chmutov, S. Duzhin, J. Mostovoy, \textit{Introduction to Vassiliev knot invariants.} Cambridge Universtity Press, 2012.

\bibitem{ChL} S.V. Chmutov, S.K. Lando, \emph{Mutant knots and intersection graphs.} Algebraic and Geometric topology, \textbf{7} (2007) 1579--1598.

\bibitem{goryunov} V.V.Goryunov, 
\emph{Finite order invariants of framed knots in a solid torus and in Arnold's (J+)-theory of plane curves,} Lecture Notes in Pure and Applied Mathematics, vol.184 (1997) `Geometry and Physics' (J.E.Andersen, J.Dupont, H.Pedersen and A.Swann, eds.), Marcel Dekker, Inc., New York--Basel--Hong Kong, 549-556.

\bibitem{Hill} J.W. Hill, \emph{Vassiliev-type invariants fo planar fronts without dangerous self-tangencies.} C. R. Acad. Sci. Paris, t. 324, S\'erie I, p. 537--542, 1997.

\bibitem{KS} V. Kleptsyn, E. Smirnov, \emph{Plane curves and bialgebra of Lagrangian subspaces.} Preprint. ArXiv:1401.6160.

\bibitem{LandoShkol} S.K. Lando, \emph{Invarianty graphov svyazannye s invariantami uzlov. Issledovatel'skie voprosy dlya shkol'nikov.} (in Russian). Available online at http://www.mccme.ru/mmks/mar08/Lando.pdf.

\bibitem{Lando} S.K. Lando, \textit{J-invariants of plane curves and framed chord diagrams.} Functional Analysis and Its Applications
January 2006, Volume 40, Issue 1, pp 1-10.

\bibitem{LandoHopf} S.K. Lando, \emph{On a Hopf Algebra in Graph Theory.} J. of Comb. Th., Series B \textbf{80}, 104-121 (2000).

\bibitem{LS} R.G. Larson, M. Sweedler, \textit{An associative orthogonal bilinear form for Hopf algebras.} Amer. J. Math. \textbf{91}, (73-93), 1969.

\bibitem{DistMutKnots} J. Murakami, \emph{Finite type invariants detecting the mutant knots.} Knot Theory, A volume dedicated to Professor Kunio Murasugi for his 70th birthday. Edited by M. Sakuma et al., Osaka University, March 2000.

\bibitem{Sage} \emph{Sage}, Open source mathematical software. www.sagemath.org

\bibitem{Sloane} \emph{Online encyclopedia of integer sequences}, Online at http://www.oeis.org
\end{document}